\documentclass{article}
\usepackage[utf8]{inputenc}
\usepackage{amsthm}
\usepackage{amsmath}
\usepackage{amsfonts}
\usepackage{color}
\usepackage{xcolor}
\usepackage{latexsym}
\usepackage{geometry}
\usepackage{enumerate}
\usepackage[shortlabels]{enumitem}
\usepackage{csquotes}
\usepackage{graphicx}
\usepackage{comment}
\usepackage{appendix}
\usepackage{hyperref}
\hypersetup{
    colorlinks=true,
    linkcolor=blue,
    filecolor=magenta,      
    urlcolor=cyan,
}
\usepackage{bm}
\usepackage{amssymb}

\emergencystretch=\maxdimen
\def \NN {\mathcal{N}}

\def \IR {\mathbb{R}}
\def \IS {\mathbb{S}}

\def \Bry {M_{\rm Bry}}

\def \eps {\epsilon}
\def \IB {\mathbb{B}}

\def \LL {\mathcal{L}}
\def \HH {\mathcal{H}}

\DeclareMathOperator{\Rm}{Rm}
\DeclareMathOperator{\Ric}{Ric}

\makeatletter
\newcommand*{\rom}[1]{\rm {\expandafter\@slowromancap\romannumeral #1@}}
\makeatother

\def\XXint#1#2#3{{\setbox0=\hbox{$#1{#2#3}{\int}$ }
\vcenter{\hbox{$#2#3$ }}\kern-.6\wd0}}

\protected\def\vts{%
  \ifmmode
    \mskip0.5\thinmuskip
  \else
    \ifhmode
      \kern0.08334em
    \fi
  \fi
}

\numberwithin{equation}{section}
\theoremstyle{plain}
\newtheorem{Theorem}[equation]{Theorem}
\newtheorem{Remark}[equation]{Remark}
\newtheorem{Proposition}[equation]{Proposition}
\newtheorem{Lemma}[equation]{Lemma}
\newtheorem{Corollary}[equation]{Corollary}
\newtheorem{Conjecture}[equation]{Conjecture}
\theoremstyle{definition}
\newtheorem{Definition}[equation]{Definition}

\def \NN {\mathcal{N}}

\def \IS {\mathbb{S}}

\def \O {\mathcal{O}}

\def \tg {\tilde{g}}

\def \MM {\mathcal{M}}

\def \rd {\mathrm{d}}

\def \orng {\color{orange}}

\title{Unique Asymptotics of Steady Ricci Solitons with Symmetry}
\author{Zilu Ma, Hamidreza Mahmoudian, and Nata\v{s}a \v{S}e\v{s}um}
\date{}

\def \O {\mathrm{O}}

\begin{document}


\maketitle
\begin{abstract}
In this paper we study 4d   gradient steady Ricci solitons, which are weak $\kappa$-solutions, and admit $\O(3)$-symmetry. Under a weak curvature decay condition, we find precise geometric asymptotics of such solitons, which are similar to those  for 3d compact $\kappa$-solutions found in \cite{ABDS22}. This is the first step towards the classification of 4d  gradient steady  Ricci solitons and more general ancient Ricci flows.
\end{abstract}

\bigskip 

\section{Introduction}
\label{sec-intro}

We call a triple $(M^4,g,f)$ a complete  gradient steady Ricci soliton,
if $(M^4,g)$ is a complete Riemannian manifold and
\[
    \Ric=\nabla^2f.
\]
By a classical result due to Hamilton, we may normalize the metric so that
\[
    R+|\nabla f|^2=1.
\]
If $(\Phi_t)_{t\in \IR}$ denotes the $1$-parameter group of diffeomorphisms generated by $-\nabla f,$ then $(M^4,g_t)_{t\in\IR}$ is a Ricci flow, called the canonical form induced by the steady soliton $(M^4,g,f).$
 
{ Steady solitons} model Type II singularities under Ricci flow (\cite{Ham93e}, \cite{Bre20}). Recall that Hamilton (\cite{Ham93}) distinguished between type I and type II finite time singularities at time $T$, depending on whether $(T-t)\, |{\Rm}|$ stays uniformly bounded, or not. By \cite{EMT11} it is known that type I singularities are modeled on shrinkers, and in general they are easier to analyze than type II singularities. In an attempt to better understand type II singularities, steady Ricci solitons received considerable attention lately, but no classification of { steady  solitons} in general case has been obtained in dimensions $n \ge 4$. 

Let us recall the notion of $\kappa$-solutions as introduced by Perelman in \cite{Per02}.

\begin{Definition}
\label{def-kappa}
 A $\kappa$-solution is an ancient solution of the Ricci flow, which has nonnegative curvature operator, positive scalar curvature, bounded curvature on compact time intervals and is $\kappa$-noncollapsed on all scales. 
\end{Definition}
 
Note that $\kappa$-solutions arise as blowup limits near cylindrical singularities. They played an important role in Perelman's proof of the geometrization conjecture in \cite{Per03a} and \cite{Per03b}. The complete classification of these solutions in 3d Ricci flow has been recently obtained in \cite{Bre20} and \cite{BDS21}, which are, up to scaling and finite quotients, the round $\IS^3$, the round shrinking $\mathbb{R}\times \IS^2$, the 3d Bryant soliton, or the 3d Perelman oval.

In this paper, we shall weaken the notion of $\kappa$-solutions and in turn study the following weak $\kappa$-solutions.
\begin{Definition}
    A \textbf{weak $\kappa$-solution} is an ancient solution of the Ricci flow, which has nonnegative sectional curvature, positive scalar curvature, bounded curvature on compact time intervals and is $\kappa$-noncollapsed on all scales. 
\end{Definition}

Let us focus now on dimension $n = 4$. In \cite{Bre14} it has been shown that a steady gradient Ricci soliton of dimension $n \ge 4$, which has positive sectional curvature, and is asymptotically cylindrical is the Bryant soliton. Known examples of 4d steady Ricci solitons up to now include the $\O(4)$-symmetric 4d steady soliton constructed by Bryant, the 3d Bryant soliton times a line, and the 1-parameter family of $\mathbb{Z}_2\times \O(3)$  symmetric steady solitons constructed by Lai (\cite{Lai20}). Solutions found by Lai are the analogous of oval bowl translators in the mean curvature flow constructed by Hoffman-Ilmanen-Martin-White. Note that the 4d Bryant soliton has a round cylinder, $\IS^3\times \mathbb{R}$ as its tangent flow, whereas the last two examples have a bubble sheet $\IS^2\times \mathbb{R}^2$ as their tangent flows. Inspired by an analogous classification of noncollapsed translators in $\mathbb{R}^4$  (\cite{CHH23}) one conjectures (see also \cite{Has23}) the following.

\begin{Conjecture}
\label{conj-main}
Any {weak} $\kappa$-solution in 4d Ricci flow, that is  also a steady Ricci soliton, is up to scaling and finite quotients, given by one of the following: 4d Bryant soliton, the 3d Bryant soliton times a line, or belongs to the 1-parameter family of $\mathbb{Z}_2\times \O(3)$-symmetric  steady solitons constructed by Lai (\cite{Lai20}).
\end{Conjecture}

\begin{Remark}
\label{rem-tan-flows}
By the work in \cite{BCDMZ21}, see Theorem \ref{thm: tan flow oo} below, gradient steady Ricci  soliton $(M^4,g,f)$, that also induces a { weak} $\kappa$-solution, has a unique tangent flow at infinity.  By \cite[Corollary 5.4]{CMZ23}, if $\Ric>0$, the unique tangent flow at infinity is either $\IS^3\times\mathbb{R}$ or $\IS^2\times\mathbb{R}^2$.
In the former case,  by \cite[Theorem 1.2]{Bre14} combined with \cite[Theorem 1.5]{CMZ23}, we know $(M^4,g,f)$ is isometric to the Bryant soliton, c.f. \cite[Theorem 1.5]{DZ20a} and \cite[Corollary 0.3]{ZZ23}. Hence, we can assume from now on that the Ricci gradient steady soliton we consider here has $\IS^2\times\IR^2$ as the tangent flow at infinity. 
We also remark that by the work in \cite{CMZ21} this tangent flow at infinity (in Bamler's sense) is identical to Perelman's asymptotic shrinker, since $\IS^2\times \IR^2$ is smooth.
\end{Remark}

{
\begin{Remark}
\label{rmk: nclp}
As Perelman mentioned in \cite[\S 11.1]{Per02} which was proved by Zhang in \cite{Z20}, for an ancient flow with $\Rm\ge 0,$ $\kappa$-noncollapsedness is equivalent to bounded Nash entropy. See also \cite[Theorem 1.13]{MZ21}, where  for steady solitons that have positive sectional curvature,  $\kappa$-noncollapsedness is equivalent to bounded Nash entropy.  Therefore, in our setting instead of assuming a soliton being $\kappa$-noncollapsed, we may also assume that the  canonical Ricci flow induced by $(M^4,g,f)$ is noncollapsed, i.e.,
\[
    \mu_\infty := \inf_{\tau>0} \NN_{o,0}(\tau) >-\infty,
\]
where $\NN_{o,0}(\tau)$ denotes the pointed Nash entropy based at a space-time point $(o,0)$ at scale $\tau$, see, e.g., \cite{Bam20a}.
\end{Remark}}

{\bf Assumptions:} In what follows we will consider GSRS (gradient steady Ricci soliton), $(M^4,g,f)$ that is as well a weak $\kappa$-solution, and also has
\begin{enumerate}
\item[({\bf A1})]$\Ric > 0$.
\item[({\bf A2})] $(M^4,g)$ admits an $\O(3)$-symmetry, that is, $\O(3) \subset \rm Isom(M^4,g)$.
\item[({\bf A3})] A soliton potential function $f$ has a critical point.
\end{enumerate}

Next, denote by 
\[\MM :=\{\mbox{GSRS}\,\, (M,g,f) \,\, \mbox{that { induces a weak}}\,\,\,\kappa\,\mbox{-solution, and satisfies assumptions {\bf (A1)-(A3)}}\}.\]
By Remark \ref{rem-tan-flows} we may assume that tangent flow of any solution in $\mathcal{M}$ is $\IS^2\times\mathbb{R}^2$.

Note that ({\bf A1}) implies a critical point of $f$ must be unique, and we denote this point by $o$. { (This holds since $\nabla^2f=\Ric>0$ and $f$ is strictly convex.) }  It is also known that the level sets $\Sigma_s := \{f = s\}$ are 3d compact manifolds{, which are in fact diffeomorphic to $\IS^3$ by ({\bf A1}). See, for example, \cite[Lemma 2.3]{DZ21}.} 


{
The analysis is somewhat similar to \cite{ABDS22}. We describe our approach here in order to make the comparison with \cite{ABDS22} easier. The main idea is that when we compare different level sets $\Sigma_s$ for large $s$, the induced metric changes in a way that closely resembles Ricci flow.
}
Consider the level sets  $\Sigma_s$, for $s \ge s_0$, where $s_0$ is a large number that will be chosen later. Denote by $\Sigma := \{f = s_0\}$. Let $\chi_s$ be a family of diffeomorphisms satisfying
\[
    \partial_s \chi_s
    = \frac{\nabla f}{|\nabla f|^2},\quad
    { \chi_{s_0}={\rm id}_{\Sigma}.}
\]
Clearly, { $\Sigma_{s}=\{f=s\} = \chi_{s}(\Sigma)$ for $s\ge s_0.$}
We define 
\begin{equation}
\label{eq: bar g_s def}
\bar g_s := \chi_{s}^*(g|_{\Sigma_{s}}),
\end{equation}
for $s\ge 0$.
Then $(\Sigma,\bar g_s)$ is isometric to $(\Sigma_{s},g|_{\Sigma_{s}})$, for $s\ge s_0.$ So,
\[
    \partial_s \bar g_s
    =  \chi_s^*\left(\frac{1}{|\nabla f|^2}\left(\mathcal{L}_{\nabla f} g\right)|_{\Sigma_{s}}\right)
    = \frac{2\chi_s^*(\Ric|_{\Sigma_{s}})}{\chi_s^*|\nabla f|^2}.
\]
Thus, we may write 
\begin{equation}
\label{eq-arf-level-set}
    \partial_s \bar g_s
    = 2 \Ric_{\bar g_s}
    + 2\bar{\mathcal{E}}(s),
    \end{equation}
where
\begin{equation}
\label{eq-error}
    \chi_{s*}\bar{\mathcal{E}}(s) = 
    \frac{\Ric|_{\Sigma_{s}}}{|\nabla f|^2}
    - \overline{\Ric}.
\end{equation}
Here and in the following, we denote by $\overline{\Ric},\bar R$ the curvatures of the induced metric $g|_{\Sigma_s}.$ 

Our goal in this paper is to complete the first step needed in completing Conjecture \ref{conj-main}, in the case the tangent flow is the bubble sheet. More precisely, we describe precise asymptotics of $(M^4,g,f)$. In other words, our result shows that under suitable scalar curvature decay, the level sets $\Sigma_s$ have similar matched asymptotics as the ones of Perelman's 3d solution (\cite{BDS21}). We actually establish unique asymptotics for the profile function of  level sets, which in the light of our assumption {\bf (A2)} are $\O(3)$ symmetric (see Lemma \ref{lem: level set sym} below). 

\begin{Theorem}
\label{thm-main}
Let $(M^4,g,f) \in \MM$. Assume in addition that
\[\sup_{\partial B_r(o)} R = O(r^{-\eta}),\]
for some $\eta\in { (\frac{1}{3},1)}$.
Then there exists a reference point $q\in \Sigma$ so that the following holds. Let $F(z,s)$ denote the radius of the sphere of symmetry in $(\Sigma, \bar{g}_s)$, which has signed distance $z$ from the reference point $q$. Then the  profile function  $F(z,s)$ has the following asymptotic expansions.
\begin{enumerate}
\item[(i)]
Fix a large number $L$. Then, as $\tau\to -\infty$,
\[F(z,s) = \sqrt{2s} - \frac{1}{4\sqrt{2} \log s \sqrt{s}}\, (z^2 - 2s)  + o\Big(s^{-1/2}(\log s)^{-1}\Big),\]
for $|z| \le L\, \sqrt{s}$.
\item[(ii)]
Fix a small number $\theta \in (0,\frac12)$, and a large number $M \ge 20$. If $s \ge s(\theta,M)$, then we have
\[2s - \frac{M^2+C(\theta)}{M^2-2}\, \frac{z^2}{2\log s} \le F^2(z,s) \le 2s - \frac{M^2-C(\theta)}{M^2}\, \frac{z^2}{2\log s},\]
whenever $z \ge M\sqrt{s}$ and $F(z,s) \ge \theta \, \sqrt{2s}$.
\item[(iii)]
The reference point $q$ has distance $(2+o(1)) \sqrt{s \log s}$ from each tip of $(\Sigma,\bar{g}_s)$. The scalar curvature at each tip is given by $(1+o(1)) \, \frac{\log s}{s}$. Finally, if we rescale the solution around one of the tips, then the rescaled solutions converge to the Bryant soliton as $s \to \infty$.
\end{enumerate}
\end{Theorem}

Without the scalar curvature decay assumption asymptotics could look quite differently from what is stated in Theorem \ref{thm-main}, and we describe those in Section \ref{sec-ref-symm}. Nevertheless, we still expect Theorem \ref{thm-main} to hold even without the scalar curvature decay assumption. { Based on Brendle's work in \cite{Bre14}, Deng and Zhu in \cite{DZ20a} proved that 4d Bryant is the only positively curved steady soliton with linear curvature decay, i.e., $\eta=1$. Our work seems to be the first to analyze steady solitons with slower curvature decay.}

Organization of the paper is as follows. In Section \ref{sec-prelim} we prove some geometric facts about our soliton, and also analyze the structure of level sets, which up to an error, turn out to have similar behavior to 3d Perelman's solution. In Section \ref{sec-parabolic} we study more carefully asymptotics of the level sets in the cylindrical region. We adopt methods from \cite{Bre20}  and \cite{ABDS22} to show that we have dichotomy in asymptotic behaviors of projections of our soliton onto eigenspaces of linearized Ricci flow operator around the cylinder. In Section \ref{sec-rulling-out} we show that under additional curvature decay assumption we have unique asymptotic behavior of our soliton in the cylindrical region. In Section \ref{sec-neutral} we give precise matched asymtotics of our solution in the case we have curvature decay assumption, and hence neutral mode dominating. In Section \ref{sec-ref-symm} we drop the assumption on scalar curvature decay, and instead impose the assumption on reflection symmetry. Without the scalar curvature decay assumption we are not able to rule out the case of positive modes dominating in dichotomy result stated in Proposition \ref{prop-merlezaag}. In the case of positive modes dominating we give precise asymptotics of our solution in the cylindrical region, hoping that can be used to eventually rule out positive mode dominating at all.

\bigskip

\noindent\textbf{Conventions.} For constants, we shall follow conventions similar to those in \cite[2.1]{Bam20a}. Throughout our paper, constants $c,C,$ depend on the steady soliton that we consider. For example, constants of the form $C(a,b)$ depend on $a,b$, and the steady soliton in question.

We write ``if $\eps\le \bar\eps(a,b),\cdots$'' to mean that ``for any $a,b,$ there is a small constant $\bar\eps(a,b)>0$ depending on $a,b$, and the steady soliton in question, such that if $\eps\le \bar\eps(a,b),\cdots$''.  Similarly, we write ``if $s\ge \underline{s}(a,b),\cdots$''.

\section{Preliminaries}

\label{sec-prelim}

\subsection{Coarse estimates}
Let $(M^4,g,f)$ be a steady gradient Ricci soliton inducing a weak $\kappa$-solution. By rescaling $g$ if necessary, we can assume the normalization $$R + |\nabla f|^2 = 1.$$ Since $M$ is nonnegatively curved, we have
\[
    0\le \sec \le R \le 1,
\]
Thanks to technical justifications made in \cite[Appendix A]{Bam21}, the theory developed by Bamler in \cite{Bam20a} and \cite{Bam20c} applies for complete Ricci flows with bounded curvature on compact intervals. Thus, we may use the following result in \cite{BCDMZ21}.

\begin{Theorem}[{\cite{BCDMZ21}}]
\label{thm: tan flow oo}
    Let $(M^4,g,f)$ be a complete nontrivial and noncollapsed steady gradient Ricci soliton with bounded curvature. Then its tangent flow at inifinity is unique and is one of the following:
    \[
        (\IS^3/\Gamma)\times \IR,\quad
        \IS^2\times\IR^2,\quad
        (\IS^2\times_{\mathbb{Z}_2} \IR)\times \IR.
    \]
\end{Theorem}

As discussed in Section \ref{sec-intro}, we may assume that the steady soliton has $\IS^2\times\IR^2$ as its unique tangent flow at infinity. 

We will use a splitting result for $4$-dimesional steady solitons proven in \cite{CMZ23}. We make the following definition before stating it.

\begin{Definition}[Dimension Reduction]
Let $(M^n,g)$ be a complete Riemannian manifold. We say $(M^n,g)$ \textit{dimension reduces} to $(N^{n-1},h)$ along some sequence $x_i\to \infty$, if (after possibly passing to a subsequence)
\[
    (M^n, R(x_i)g, x_i) \to (N^{n-1}\times \IR, h+{\rm d}z^2, (x_\infty,0)),
\]
locally smoothly in the sense of Cheeger-Gromov convergence.
\end{Definition}

\begin{Theorem}[{\cite[Theorem 1.2]{CMZ23}}]
\label{thm: dim red}
    Let $(M^4,g,f)$ be a complete noncollapsed steady soliton with $\sec\ge 0,\Ric>0$.
    Suppose $M$ has $\IS^2\times \IR^2$ as its tangent flow at infinity.
    Then along any sequence $x_i\to \infty$, $(M,g)$ dimension reduces to either $M^3_{\rm Bry}$, the 3d Bryant's steady soliton, or $\IS^2\times \IR$, the round cylinder with scalar curvature $1.$
\end{Theorem}


Now assume $(M,g,f) \in \MM$ is a steady soliton satisfying our assumptions. For $x\in M$, let $r_{\rm sym}(x)$ denote the radius of the symmetric $\IS^2$ passing through $x$. Using the Theorem above we can show the following.

\begin{Lemma}
\label{prop: orb sec dominates}
    There is a constant $C$ depending on $(M^4,g,f) \in \MM$ such that 
    \[
         R(x) \le Cr_{\rm sym}^{-2}(x),
    \]
    for any $x\in M.$
\end{Lemma}
\begin{proof}
    Suppose there exist $x_i\to \infty$ such that
    $$ R(x_i)r_{\rm sym}^2(x_i) \to \infty.$$
    By Theorem \ref{thm: dim red},  $(M, R(x_i)g, x_i)$ converges to some limit pointed manifold $(M_\infty, g_\infty, x_\infty)$, locally smoothly.  Here, $(M_\infty, g_\infty)$ is either $M^3_{\rm Bry}\times \IR$ or $\IS^2\times \IR^2$. Along the convergence,
    \[
        r_{\rm sym}^2(x_i) R(x_i)
        \to r_{\rm sym}^2(x_\infty) < \infty,
    \]
    which is a contradiction. 
\end{proof}

We will also need the following version of Perelman's result (c.f. \cite[Corollary 42.1-(4)]{KL08}) proven in the appendix of \cite{CMZ23}. As mentioned in Remark \ref{rmk: nclp}, the assumption of bounded entropy can be replaced by being $\kappa$-noncollapsed by \cite[Theorem 1.9]{MZ21}.

\begin{Theorem}[{\cite[Corollary { A.9}]{CMZ23}}]\label{perelmanbds}
\label{thm: Rm der est}
    Let $(M^n,g_t)_{t\le 0}$ be a complete ancient solution to Ricci flow with Nash entropy bounded from below by $\mu_\infty$. Suppose $(M^n,g_t)_{t\le 0}$ has bounded curvature on compact time-intervals, and assume
    \[
        \Ric\ge 0 \, , \quad
        |{\Rm}|\le \Lambda R
    \]
    on $M\times (-\infty,0]$ for some constant $\Lambda<\infty$. Then for any $k\ge 0$, there is a constant $C_k=C(n,\mu_\infty,\Lambda,k)<\infty$ such that
    \[
        |\nabla^k{\Rm}|
        \le C_k R^{\frac{k+2}{2}}
    \]
    on $M\times (-\infty,0]$.
\end{Theorem}

\subsection{Structure of the level sets}

The goal of this subsection is to describe some geometric properties of the level sets of the potential function $f$. We show that each level set is rotationally symmetric, asymptotically cylindrical and has exactly two tips resembling the Bryant soliton. Throughout this subsection, we assume $(M^4,g,f)\in \MM$ unless otherwise stated. {For simplicity, we also assume that the scalar curvature decays uniformly.}

In the Lemma that follows we show that under our Assumptions {\bf(A1)-\bf(A3)}, the level sets $\Sigma_s=\{f=s\}$ of $f$ are rotationally symmetric. 

\begin{Lemma}
\label{lem: level set sym}
    Let $o$ be a critical point of the potential $f$. For any $\phi\in {\rm Isom}(g),$
    \[
    \phi(o)=o, \quad \phi^*f=f.
    \]
    As a consequence, for any level set $\Sigma_s$ with the induced metric, 
    \[
        {\rm Isom}(g)|_{\Sigma_s}\subseteq {\rm Isom}(\Sigma_s).
    \]
\end{Lemma}
\begin{proof}
    By Assumptions {\bf (A1)} and {\bf (A3)}, the point $o$ is the unique point with $R(o)=1$. Since we have $\phi^*R=R$, we must have $ \phi(o)=o$. Note also that $\phi^*f$ is a potential function for $(M,g)$. Hence $\nabla^2(\phi^*f-f)=0$ and $\phi^*f=f+C$ for some $C$, otherwise $(M,g)$ locally splits off a line in the direction of $\nabla(\phi^*f-f)$ which contradicts ({\bf A1}). Finally we have
    \[
        C= (\phi^*f)(o)-f(o)=0.
    \]
\end{proof}

We now discuss asymptotic cylindricity of the levels sets, in the following sense.
\begin{Definition}[$(\eps,k)$-centers] 
Let $(M^n,g)$ be a Riemannian manifold. We say that a point $x\in M$ is an $(\eps,k)$-center, or that $M$ is $\eps$-close to $\IS^{n-k}\times \IR^k$ at $x$, if there is a smooth embedding $\Phi:V \to M$ with $B((\bar x,0), 1/\eps)\subseteq V\subseteq \IS^{n-k}\times \IR^{k}$ and $\Phi(\bar x,0)=x$, such that
for $\tilde g = R(x)g$ we have
    \[
    \left\|\Phi^*\tilde g - \bar g\right\|_{C^{[1/\eps]}}  < \eps,
    \]
where $\bar g$ denotes the standard round metric on $\IS^m\times \IR^{k-m}$ with constant scalar curvature $1.$
\end{Definition}

\begin{Lemma}
\label{lemma-parabolic}
    For any $\eps>0$, there exists a $\delta(\epsilon) > 0$ such that if $x$ is a $(\delta,2)$-center with $\delta\le\delta(\eps)$, then $\Sigma_{f(x)}$ is $\eps$-close to $\IS^2\times \IR$ at $x$.
\end{Lemma}
\begin{proof}
    Suppose not. Then there exists $\eps>0$, $\delta_i\to 0$, $x_i\in M$ such that $x_i$ is a $(\delta_i,2)$-center, but $\Sigma_{f(x_i)}$ is not $\eps$-close to $\IS^2\times \IR$ at $x_i$. Let $r_i= R^{-\frac{1}{2}}(x_i)$ and $g_i=r_i^{-2}g$. Since $x_i$ are $(\delta_i,2)$-necks and $\delta_i \to 0$, Theorem \ref{thm: dim red} implies $$(M, r_i^{-2}g, x_i) \to \IS^2\times \IR^2.$$
    By {\cite[Lemma 4.3]{CMZ23}} we have $r_i\to \infty$ and
    \[
        (f-f(x_i))/r_i \to z,
    \]
    where $z$ is the coordinate function of some $\IR$-factor. By rotation, we may assume $z=z_2$,
    where $(z_1,z_2)$ are the coordinates of the $\IR^2$-factor in $\IS^2\times \IR^2$. 
    Thus, $(\Sigma_{f(x_i)}, r_i^{-2}g,x_i )$ converges to $\IS^2\times \IR$.
    This is a contradiction.
\end{proof}

Let
\[
    0<\lambda_1\le \cdots
    \le \lambda_4
\]
denote the eigenvalues of $\Ric.$

\begin{Lemma}
\label{lem: theta}
For any $\epsilon>0$, there is a $\theta=\theta(\epsilon)\in (0,1)$ such that $\lambda_2(x)\le \theta R(x)$ implies $x$ is an $(\epsilon,2)$-center and $\Sigma_{f(x)}$ is $\eps$-close to $\IS^2\times \IR$ at $x$.
\end{Lemma}
\begin{proof}
We argue by contradiction. Suppose there exists an $\epsilon>0$ and $x_i\to \infty$ such that $\lambda_2(x_i)/R(x_i) \to 0$ but $x_i$ are not $(\epsilon,2)$-centers. By Theorem \ref{thm: dim red}, $M$ dimension reduces to $\Bry^3$ along $x_i$ (since $x_i$ are not $(\epsilon,2)$-centers), which contradicts to the fact that $\lambda_2(x_i)/R(x_i) \to 0$. By reducing $\eps$ is necessary, we can use Lemma \ref{lemma-parabolic} to conclude $\eps$-closeness of $\Sigma_{f(x)}$.
\end{proof}

The next few results demonstrate that for large $s$, the level sets $\Sigma_s$ look similar to compact ancient 3d solutions.

\begin{Proposition}
\label{prop: disjoint Bry}
Consider an arbitrary sequence $s_k\to \infty$. For $k\ge \underline{k}$, there are two disjoint compact domains $\Omega^1_k \, ,\, \Omega^2_k\subset \Sigma_{s_k}$ with the following properties.
\begin{itemize}
    \item $\Omega^1_k,\Omega^2_k$ are diffeomorphic to $\IB^3.$
    \item $ \lambda_2(x)<\theta R(x)$ for any $x\in \Sigma_{s_k}\setminus \Omega_k
    $ ,where $\Omega_k=\Omega^1_k\cup\Omega^2_k$. 
    \item $\lambda_2(x)>\frac{1}{2}\theta R(x),$ for any $x\in \Omega_k.$
    \item $\partial \Omega_k$ are two leaves of { Hamilton's CMC foliation} in $\Sigma_{s_k}$.
    \item There is a leaf $\Gamma_k$ of the CMC foliation in $\Sigma_{s_k},$ such that 
    $\Omega^1_k\,, \, \Omega^2_k$ lie in the two different components of $\Sigma_{s_k}\setminus \Gamma_k,$ and 
    $
       \sup_{\Gamma_k} \frac{\lambda_2}{R}\to 0.$
    \item For $i=1,2$ and for any sequence of points $x^i_{k}\in \Omega^i_{k}$ , $\big(\Omega^i_{k} \, , \, R(x^i_{k})g\, , \, x^i_{k}\big)$ converges to a bounded domain in $\Bry^3$ which contains the tip of $\Bry^3.$
\end{itemize}
\end{Proposition}
\begin{proof}
The proof is almost the same as  that of \cite[Proposition 2.4]{BDS21} now that we have \cite[Theorem 1.3]{CMZ23}.
\end{proof}

For a fixed $s$, we let $\overline{\nabla}$ denote the projection of $\nabla$ onto the tangent space of $\Sigma_{s}$.

\begin{Definition}
\label{def-tip}
We say that $p\in \Sigma_s$ is a \textbf{tip} of $\Sigma_s$, if 
 $\lambda_2(p)> R(p)/6$ and $\overline{\nabla}R(p)=0$.
\end{Definition}

\begin{Corollary}
\label{cor-b1}
Consider an arbitrary sequence $s_k\to \infty$. Let  $\Omega^1_k,\, \Omega^2_k\subset \Sigma_{s_k}$ be the two disjoint compact domains given by Proposition \ref{prop: disjoint Bry}. For $k\ge \underline{k}$, there are exactly two tips in $\Sigma_{s_k}$, one in $\Omega^1_k$ and the other one in $\Omega^2_k$.
\end{Corollary}

\begin{proof}
The proof is almost the same as  that of \cite[Corollary 2.5]{BDS21}.
\end{proof}

\begin{Proposition}
\label{prop-b2}
Consider an arbitrary sequence $s_k\to \infty$.
Let $p^1_k\, , \, p^2_k$ be the two tips in $\Sigma_{s_k}$ for $k\ge \underline{k}$. Then for $i=1,2$, as $k\to \infty$, we have
\[
    (\Sigma_{s_k},\, R(p^i_k)g \, ,\, p^i_k) \to \Bry^3.
\]
\end{Proposition}

\begin{proof}
The proof is almost the same as  that of \cite[Proposition 2.6]{BDS21}.
\end{proof}

\begin{Proposition}
\label{prop-b3}
Consider an arbitrary sequence $s_k\to \infty$.
Let $p^1_k\, ,\, p^2_k$ be the two tips in $\Sigma_{s_k}$ for $k\ge \underline{k}.$
Then for $i=1,2,$ as $k\to \infty,$
\[
    R(p^i_k) d(p^1_k\, , \, p^2_k)^2 \to \infty.
\]
\end{Proposition}

\begin{proof}
The proof is almost the same as  that of \cite[Proposition 2.7]{BDS21}.
\end{proof}

\begin{Proposition}
\label{prop-b4}
Consider an arbitrary sequence $s_k\to \infty$ and let $x_k\in \Sigma_{s_k}$.
Let $p^1_k\, ,\, p^2_k$ be the  two tips in $\Sigma_{s_k}$ for $k\ge \underline{k}$.
If
    $R(p^1_k) d_g(x_k\, ,\, p^1_k)^2 \to \infty$ and $R(p^2_k) d_g(x_k\, ,\, p^2_k)^2 \to \infty$
then
    $\tfrac{\lambda_2}{R}(x_k) \to 0$.
\end{Proposition}

\begin{proof}
The proof is almost the same as  that of \cite[Proposition 2.8]{BDS21}.
\end{proof}

By combining Corollary \ref{cor-b1}, Proposition \ref{prop-b2}, Proposition \ref{prop-b3}, and Proposition \ref{prop-b4} we can draw the following conclusion.


\begin{Corollary}
\label{cor: qualitative}
\begin{enumerate}[start=1,label={\rm (\roman*)}]
         \item If $s\ge \underline{s}, \Sigma_s$ has exactly two tips, denoted by $p^1_s,p^2_s$, and they vary continuously in $s.$ We denote by
         \[
            r_{i,s}:= R(p^i_s)^{-\frac{1}{2}}
         \]
         the scalar curvature radii at the tips, $i=1,2.$ Then
         \begin{equation}
         \label{eq: tip curv coarse}
            \lim_{s\to \infty} \frac{r_{i,s}}{\sqrt{s}} \to 0.
         \end{equation}
         \item For any $A<\infty,$ if $s\ge \underline{s}(A),$
         \[
            B(p^1_s,Ar_{1,s}) \cap B(p^2_s,Ar_{2,s}) = \emptyset.
         \]
         \item For any $A<\infty,\eps>0,$ if $s\ge \underline{s}(A,\eps),$ 
         $\big(B(p^i_s,Ar_{i,s}), r_{i,s}^{-2}g, p^i_s\big)$ is $\eps$-close to the corresponding piece in $\Bry^3\times \IR, i=1,2.$
         \item For any $\eps>0,$ there are $\underline{s}(\eps), A(\eps)<\infty$ with the following property. If $s\ge \underline{s},$ and $x\notin B(p^1_s,Ar_{1,s}) \cup B(p^2_s,Ar_{2,s}),$ then $x$ is an $\eps$-center.
     \end{enumerate}     
\end{Corollary}

\begin{Proposition}[{\cite[Proposition 2.8]{ABDS22}}]
\label{prop: neck lemma}
Let $(M^4,g,f)\in \MM$. For any $\eps>0$ and $\delta>0$, there exists $\underline{s}(\eps,\delta)$ with the following property: If $s\ge \underline{s}(\eps,\delta)$ and $x\in \Sigma_s$ satisfies
$r_{\rm sym}(x)\ge \delta\sqrt{s}$, then $x$ is an $\eps$-center.
\end{Proposition}
\begin{proof}
    The proof is almost the same as that of \cite[Proposition 2.8]{ABDS22}, and uses Theorem \ref{thm: dim red} and \eqref{eq: tip curv coarse}.
\end{proof}

{

\subsection{Brendle's barrier function}

As in \cite{ABDS22}, the barrier function constructed by Brendle in \cite[Proposition 2.7]{Bre20} plays a key role in our analysis.
In order to control our error terms, we refine the statement of properties of Brendle's barrier slightly and obtain the following.

\begin{Proposition}
\label{prop: barrier spatial}
There exist universal constants $r_*\in (0,1)$, $N$ and $\underline{a} \gg 1$, and $0 < \theta \ll 1$ with the following properties: 
For $a\ge \underline{a}$, we can find a $C^1$ function 
\[
    \psi_a=\psi_a : \left[ r_*a^{-1}, 1 + \tfrac{1}{100}a^{-2}\right]
    \to \IR,
\]
satisfying
\begin{equation}
\label{eq: Bre barrier}
    \psi_a\psi_a'' - \tfrac{1}{2}\psi_a'^2
    + s^{-2}(1-\psi_a)(s\psi_a' + 2\psi_a)
    - s\psi_a' \le 
    \left\{ \begin{array}{ll}
        - \tfrac{1}{4}
        a^{-4}s^{-5},
        & \text{ on } [Na^{-1},1+\frac{1}{100}a^{-2}]\\
        -\tfrac{1}{2}a , & \text{ on }[r_*a^{-1},Na^{-1}]
    \end{array} \right.
\end{equation}   
Moreover,
\begin{itemize}
    \item $\psi_a(s)\le Ca^{-2}$ for $s\in \left[ \frac{1}{10}, 1 + \tfrac{1}{100}a^{-2}\right]$
    \item $\psi_a(s)\le 2a^{-2}s^{-2}$ for $s\in \left[ Na^{-1}, 1 + \tfrac{1}{100}a^{-2}\right]$
    \item  $\psi_a(s) \ge \frac{1}{32}a^{-4}$ for  $s\in \left[ r_*a^{-1}, 1 + \tfrac{1}{100}a^{-2}\right]$
    \item  $\psi_a(s) \ge a^{-2}(s^{-2}-1) + \frac{1}{16}a^{-4}$, for  $s\in \left[ 1-\theta, 1 + \tfrac{1}{100}a^{-2}\right]$
    \item $\psi_a(r_*a^{-1})\ge \frac{3}{2}$. 
\end{itemize}
\end{Proposition}

We shall use the following parabolic version.
\begin{Corollary}
\label{cor: spacetime barrier}
    Let $\psi_a$ be the function constructed in Proposition \ref{prop: barrier spatial}.
    Then the function
    \[
        \Psi_a(r,s) := \Psi_a(r,s):= \psi_a \left( \frac{r}{\sqrt{2s}}\right)
    \]
    satisfies
    \[
        -\Psi_{a,s}
        > \Psi_a \Psi_{a,rr} - \tfrac{1}{2}\Psi_{a,r}^2
        + r^{-2}(1-\Psi_a)(r\Psi_{a,r} + 2\Psi_a)
        + Q
    \]
    for $s\ge \underline{s}$, where
    \[
       Q\ge  \left\{ \begin{array}{ll}
              \tfrac{1}{8} a^{-4}(r/\sqrt{2s})^{-5}s^{-1}, & \text{ for } r/\sqrt{2s}\in [Na^{-1},1+\frac{1}{100}a^{-2}]\\
             \tfrac{1}{4}as^{-1}, & \text{ for }r/\sqrt{2s} \in [r_*a^{-1},Na^{-1}]
         \end{array} \right.
    \]
    Additionally, $\Psi_a$ satisfies the other properties stated in Proposition \ref{prop: barrier spatial}.
\end{Corollary}

\begin{proof}
The corollary follows immediately from Proposition \ref{prop: barrier spatial}.
\end{proof}

}

\def \bn {\mathbf{n}}
\def \bEE {\,\overline{\mathcal{E}}}
\def \EE {\,\mathcal{E}}

\section{Asymptotic analysis of level sets near the cylinder}
\label{sec-parabolic}

In this section we discuss fine asymptotic behaviors of the level sets $\Sigma_s = \{ f=s\}$ as $s \to \infty$. Let $s_0>0$ be a large number { to be determined}. We proceed to define a family of metrics on the level set $$\Sigma=\Sigma_{s_0} = \{ f=s_0\}$$
Consider the family $\chi_s$ of diffeomorphisms on our steady soliton $M$ satisfying
\[
    \partial_s \chi_s
    = \frac{\nabla f}{|\nabla f|^2},\quad
    \chi_{s_0}={\rm id}.
\]
This family preserves the level sets in the sense that $\Sigma_{s} = \chi_s(\Sigma)$. Thus we defined in \eqref{eq: bar g_s def} a metric $$\bar g_s := \chi_s^*(g|_{\Sigma_{s}})$$ on $\Sigma$ so that for $s\ge s$, $(\Sigma,\bar g_s)$ is isometric to $(\Sigma_s,g|_{\Sigma_s})$. 

Recall that if $\Phi_t$ is the flow of $-\nabla f$, then $g_t=\Phi_t^*g$ evolves by Ricci flow. Therefore, a simple computation shows
\begin{align} \label{oureq}
    \partial_s \bar g_s
    = 2 {\chi_s}^* \left(\frac{\Ric|_{\Sigma_{s}}}{|\nabla f|^2}\right) = 2 \Ric_{\bar g_s}
    + 2\bar{\mathcal{E}}(s),
\end{align}
where
\begin{align*}
    \bar{\mathcal{E}}(s) = \chi_{s*} \left( \frac{\Ric|_{\Sigma_{s}}} {|\nabla f|^2} - \Ric_{\bar g_s} \right).
\end{align*}

We shall focus on the evolution equation \eqref{oureq} for $(\Sigma, \bar g_s)$. We will show that the contribution from the error term $\bEE$ is small, and thus \eqref{oureq} can be considered as an almost ancient Ricci flow.

By Lemma \ref{lem: level set sym}, we know that $g|_{\Sigma_{s}}$ and hence $\bar g_s$ is $O(3)$-invariant. We can therefore write
\begin{equation}
\label{eq-metric}
    \bar g_s
    = {\rm d}z^2 + F^2(z,s)\, g_{\IS^2}.
\end{equation}
Here $g_{\IS^2}$ denotes the standard round metric on $\IS^2$ with radius $1$, and $z=z(s)$ represents the signed distance on $\Sigma$ from some fixed point $q$ (to be determined), measured with respect to $\bar g_s$. We can also decompose the error term $\bEE$ into its radial and orbital parts as
\begin{equation}
\label{eq-error-decomp}
    \bEE = \bEE^{\rm rad}{\rm d}z^2 + \bEE^{\rm orb}g_{\IS^2}.
\end{equation}

In the rest of this section, we consider the evolution equation for the function $F$, and derive estimates for $\bEE^{\rm rad}$ and $\bEE^{\rm orb}$. 
It will then be possible to carry out an analysis similar to \cite{ABDS22} for an appropriately chosen reference point $q$.

\subsection{Choice of Reference Point} 
As in \cite{ABDS22}, we shall first fix our reference point $q$. 
We define $\Phi_t$ by
\[
    \partial_t \Phi_t
    = -\nabla f\left( \Phi_t \right),\quad
    \Phi_0={\rm id},
\]
so that the Ricci flow solution induced by the steady soliton is $g_t=\Phi_t^*g$. As before, we shall denote the distance between $x$ and $y$ with respect to the soliton metric $g$ by $d_g(x,y).$

\begin{Proposition}[{\cite[Proposition 3.1]{ABDS22}}]
    There is a $q\in \Sigma,$ such that $$\limsup_{s\to \infty} s R(\Phi_{-s}(q)) \le 100.$$
\end{Proposition}
\begin{proof}

    The proof is similar to that of \cite[Proposition 3.1]{ABDS22} and we shall only sketch a few differences.
    Suppose on the contrary that no such $q\in \Sigma$ exists.
    By \cite[Proposition 4.5]{CMZ23}, for any $\eps>0,$ if $s\ge \underline{s}(\eps),$ there is $x_s\in \Sigma_s$ that is an $(\eps,2)$-center.
    It is not hard to see that $y_s= \Phi_{\sigma_s}(x_s)\in \Sigma$ for some $\sigma_s>0,$ by e.g. the proof of \cite[Proposition 3.1]{CMZ21}.
    Thus, for the induced Ricci flow, there exist $(q_k,s_k)\in M\times(-\infty,0)$ that are $(\eps_k,2)$-centers, where $\eps_k\to 0, q_k\in \Sigma, s_k\to-\infty.$
    The rest of the proof follows verbatim as in \cite[Proposition 3.1]{ABDS22}, 
  where Perelman's compactness theorem  and the existence of asymptotic shrinkers \cite[Section 11]{Per02} were applied for 3d $\kappa$-solutions. In our case, we cannot directly apply Perelman's original statements as they require $\Rm\ge 0$ in dimension $4$ and higher, and we shall use the version given in \cite[Theorem 1.13]{MZ21}.
\end{proof}

\begin{Proposition}[{\cite[Proposition 3.2]{ABDS22}}]
    For the induced Ricci flow $(M,g_t)_{t\in \IR},$ for any $t_k\to -\infty$, if we dilate the flow around $(q,t_k)$ by factor $(-t_k)^{-\frac{1}{2}},$ then the rescaled manifolds converge to $\IS^2\times \IR^2$ of radius $\sqrt{2}.$
\end{Proposition}
\begin{proof}
    The proof follows verbatim as that of \cite[Proposition 3.2]{ABDS22}.
\end{proof}

\begin{Proposition}
    There is a $q\in \Sigma,$ such that for any $s_k\to \infty,$
    after passing to a subsequence,
    \[
        (M,s_k^{-1}g, \chi_{s_k}(q)) 
        \to (\IS^2\times \IR^2, 2g_{\IS^2}\oplus g_{\mathbb{R}^2}, q_\infty),
    \]
    locally smoothly.
    In particular, for any $\eps>0,A<\infty, $ if $s\ge \underline{s}(\eps, A),$ $|y,\chi_s(q)| \le A\sqrt{s},$ then
    \[
        1-\eps \le s \cdot R(y) \le  1+\eps.
    \]

\end{Proposition}
{
\begin{proof}
Let $q\in \Sigma$ be chosen as above.
    Since $s\mapsto \chi_s(q)$ is a reparametrization of $t\mapsto \Phi_t(q),$ 
    for any $s\ge s_0,$ there is $t_s<0$ such that $\chi_s(q)=\Phi_{t_s}(q).$
    It suffices to show that $-t_s\sim s.$

    By the proposition above, $-t R(\Phi_t(q))\to 1,$ as $t\to -\infty.$ In particular, $|\nabla f|^2(\Phi_t(q)) = 1-R(\Phi_t(q))
    = 1-o(1).$
    For any $\eps>0,$ if $t\le \bar t=\bar t(\eps)$, then $|\nabla f|^2(\Phi_t(q))\ge 1-\eps.$
    Then for large $s,$
    \begin{align*}
        s = f(\Phi_{t_s}(q))
        &= f(\Phi_{\bar t}(q))- \int_{t_s}^{\bar t} \partial_t f(\Phi_t(q))\, {\rm d}t\\
        &= \bar s + \int_{t_s}^{\bar t} |\nabla f|^2(\Phi_t(q))\, {\rm d}t
        \ge (1-\eps)(\bar t-t_s) + \bar s,
    \end{align*}
    where $\bar s = f(\Phi_{\bar t}(q)).$ On the other hand, since $|\nabla f|\le 1,$
    \[
        s = f(\Phi_{t_s}(q))
        = f(q)+ \int_{t_s}^{0}  |\nabla f|^2(\Phi_t(q))\, {\rm d}t
        \le s_0 + |t_s|.
    \]
    Thus, $-t_s\sim s,$ and the conclusions follow.
\end{proof}
}

We may take such $q$ as a reference point in the following.

\subsection{Analysis of the error terms}


In this subsection, we estimate the size of the error terms in \eqref{eq-error-decomp} together with their derivatives.
Let $\{e_i\}_{i=1}^{4}$ be a local orthonormal frame in the region $\{x\in M: s=f(x)\ge s_0, r_{\rm sym}(x)>0\}.$
For simplicity,  assume that $e_3=\partial_z,e_4=\frac{\nabla f}{|\nabla f|},$ and $e_1,e_2$ are tangent to the symmetric sphere. 
We follow the convention $$R_{ijk\ell} = R(e_i, e_j, e_k, e_\ell), \qquad R_{ij}=\Ric(e_i, e_j), \qquad \cdots$$ and denote the quantities related to $\bar g_s$ by $\bar R$, $\bar \nabla$, etc. and those related to $g$ by $R$, $\nabla$, etc. For convenience, we shall also stop explicitly writing the differmorphism $\chi_s$, and use $\Sigma_s$ and $\bar g_s$ interchangeably to refer to either an arbitrary level set $\{f=s \}$ with the induced metric, or our fixed level set $\Sigma$ equipped with $\bar g_s$. 
We will use the estimate $F^m\partial_z^{m+1}F \leq C_m$ for $m \geq 0$ in Lemma \ref{lem: error by F}. This is proven in the next section. However, the proof only uses results of Section \ref{sec-prelim}. We have chosen to present the results in this order for the sake of clarity.

By the standard Gauss formula, we have (see, for example, \cite[(3.11)]{DZ20b})
\begin{equation}
    \label{eq: gauss}
    R_{ij}
    = \overline{R}_{ij}
    + R(\bn,e_i,e_j,\bn)
    - \frac{1}{|\nabla f|^2} 
    \sum_{k=1}^{3}
    (R_{ij}R_{kk}-R_{ik}R_{kj}).
\end{equation}
Recall that we have normalized the metric $g$ so that
   $ |\nabla f|^2 + R=1$. Hence
   \begin{align} \label{bEE explicit}
    \bEE_{ij} & = \frac{R_{ij}}{|\nabla f|^2} - \overline{R}_{ij} = \frac{1}{|\nabla f|^2} R R_{ij}+ R_{ij}-\overline{R}_{ij} \nonumber \\
    & = \frac{1}{|\nabla f|^2} R R_{ij} + R_{4ij4} - \frac{1}{|\nabla f|^2} 
    R_{ij}(R-R_{44}) + \frac{1}{|\nabla f|^2}\sum_{k=1}^{3}R_{ik}R_{kj} \nonumber\\
    & =R_{4ij4} +\frac{1}{|\nabla f|^2} 
    R_{ij}R_{44} + \frac{1}{|\nabla f|^2}\sum_{k=1}^{3}R_{ik}R_{kj}.
\end{align}

We first show that $\bEE^{\rm orb},\bEE^{\rm rad}$ have a sign in the case we consider.

\begin{Lemma}
\label{lem: pos Eorb}
Assume $(M^4,g,f) \in \MM$. Then for $s\ge \underline{s} $,
    \[
        \bEE^{\rm orb}\ge 0 \quad, \quad \bEE^{\rm rad}\ge 0.
    \]
\end{Lemma}
\begin{proof}
By the definition of $\EE^{\rm orb}$, 
\begin{align*}
    F^{-2}\bEE^{\rm orb}
    &= \bEE_{11}= R_{1441}+\frac{1}{|\nabla f|^2} R_{11}R_{44}
    +\frac{1}{|\nabla f|^2} \sum_{k=1}^3 R_{1k}^2.
\end{align*}
Similarly,
\[
    \bEE^{\rm rad} = 
    R_{3443}+\frac{1}{|\nabla f|^2} R_{33}R_{44}
    +\frac{1}{|\nabla f|^2} \sum_{k=1}^3 R_{3k}^2.
\]
Non-negative sectional curvature implies $R_{3443} \ge 0$, $R_{1441} \ge 0$, and by assumption $R_{ii} > 0$. So the assertion follows.
\end{proof}

We can also estimate the derivatives of $\bEE$.  Since $\nabla^2f=\Ric>0$, $f$ has a unique critical point. Moreover, if $\Phi_t$ denotes the flow of $-\nabla f,$
\[
    \partial_t R(\Phi_t) = -2\Ric(\nabla f,\nabla f) < 0,
\]
and thus $R<1$ away from the critical point. Hence

\begin{equation}
\label{eq-nablaf-R}
    \frac{1}{|\nabla f|^2}
    = \frac{1}{1-R}=1+O(R).
\end{equation}
In order to better understand the error terms we carefully analyze the terms in \eqref{eq: gauss}.
By \cite[Lemma 3.1]{DZ20b}, for a steady soliton we have
\begin{equation}
    \label{eq: sec grad f}
    |\nabla f|^2 
    R(\bn,e_i,e_j,\bn)
    = -\tfrac{1}{2} \nabla^2_{ij}R
    - \sum_{k=1}^{3} R_{ik}R_{kj}
    + (\Delta \Ric)_{ij}
    + 2\sum_{k,l=1}^{3}
    R_{iklj}R_{kl}
    = O(R^2).
\end{equation}
We then have the following rough estimates for $\bEE^{\rm rad}$ and $\bEE^{\rm orb}$.

\begin{Lemma}
\label{lem: error by F}
For $s\ge \underline{s}, k\ge 0,$
\[
    \left|\bEE^{\rm rad}\right| \le C F^{-4},\quad
    \left|\partial_z^k\bEE^{\rm orb}\right| \le C_{k} F^{-2-k}.
\]
\end{Lemma}
\begin{proof}

     Since $\tfrac{1}{|\nabla f|^2}=1+O(R)$ and $R \leq CF^{-2}$ on our region by Lemma \ref{prop: orb sec dominates}, the expression \eqref{bEE explicit} for $\bEE$ together with identity \eqref{eq: sec grad f} and Theorem \ref{thm: Rm der est} imply
    \begin{equation}
    \label{ineq: error bd by scal}
        |\nabla_{g}^k \bEE(s)|_g
        \le C_{k}R^{2+\frac{k}{2}}
        \le C_{k}F^{-4-k}.
    \end{equation}
    On the other hand, since  $\bar g_s=\rd z^2+F^2g_{\IS^2}$, we have
    \begin{gather*}
        \left|\bEE^{\rm rad}\right|^2
        + 2\left| F^{-2} \bEE^{\rm orb}\right|^2
        =|\bEE|_{\bar g_s}^2 \le C F^{-4},\\
        F^{-2}\left|\partial_z^k \bEE^{\rm orb}\right|
        \le C \left|\nabla_{g}^k \bEE\right|_{g}
        \le C_k R_g^{2+\frac{k}{2}}
        \le C_k F^{-4-k}.
    \end{gather*}
\end{proof}

\begin{Lemma}
    On 
 the region $\{r_{\rm sym}\ge \sqrt{s}\},$ as $s\to\infty,$
 \[
    R_{1221}=\left(\tfrac{1}{2}+o(1)\right)R,\quad
    R_{a\beta \gamma b} = o(1) R,
    \quad
    R_{ijk4} = O(R^{\frac{3}{2}}),
\]
where $a,b\in \{1,2\}, \beta,\gamma\in \{3,4\}$, and $i,j,k \in \{1,2,3,4\}$.
Consequently,
\[
    R_{11}=R_{22}=\left(\tfrac{1}{2}+o(1)\right)R,\quad
    R_{33}=o(1)R.
\]
Moreover,
\[
    R_{a44a} = \left(\tfrac{1}{4}+o(1)\right)R^2,\quad
    R_{44}=\left(\tfrac{1}{2}+o(1)\right)R^2.
\]
\end{Lemma}
\begin{proof}
    We will prove the first asymptotics statement, the proofs of the rest are similar.
    By Proposition \ref{prop: neck lemma}, for any $\eps>0,$ if $s\ge \bar s,$ and $r_{\rm sym}(x)\ge \sqrt{s},$ then $x$ is an $(\eps,2)$-center. 
    So, $|R^{-1}R_{1221}-\tfrac{1}{2}|\le \eps$.
Furthermore, since $R_{ij}=f_{ij},$
    \[
        |\nabla f| R_{ijk4}
        =R_{ijk\ell} \nabla_{\ell} f
        = -\nabla_{i}\nabla_j \nabla_k f
        + \nabla_j \nabla_i \nabla_k f
        = -\nabla_i R_{jk}+\nabla_j R_{ik}
        = O(R^{\frac{3}{2}}),
    \]
    where we have used Theorem \ref{thm: Rm der est}.
    It follows that
    \[
        R_{11}=R_{22}
        =R_{1221}+R_{1331}+R_{1441}
        = \tfrac{1}{2}R
        + o(1)R.
    \]
    By \eqref{eq: sec grad f}, \eqref{eq-nablaf-R},
    since $R_{1221}=(\tfrac{1}{2}+o(1))R,$
    \begin{align*}
        R_{1441}&= -R_{11}^2+2R_{1221}R_{22}
        + o(1)R^2
        = (-\tfrac{1}{4} + \tfrac{1}{2} + o(1)) R^2\\
        &= \tfrac{1}{4}R^2 +o(1)R^2.
    \end{align*}
\end{proof}



Finally, we refine our estimate for the error term $\bEE^{\rm orb}.$

\begin{Lemma}
\label{lem: Eorb>0}
    $\bEE^{\rm orb}(z,s) = \left(\tfrac{1}{2}+o(1)\right)F^2R^2>0,$ 
    whenever $s\ge\underline{s},$ and $F(z,s)\ge \sqrt{s}.$
\end{Lemma}
\begin{proof}
    We identify $\chi_{s*}\bEE$ and $\bEE$ when there is no ambiguity.
Recall that $R+|\nabla f|^2=1,$ and $1=\bar g_s(e_1,e_1)=F^2g_{\IS^2}(e_1,e_1).$ 
By \eqref{eq: gauss},
\begin{align*}
    F^{-2}\bEE^{\rm orb}
    &= \bEE(e_1,e_1)
    = \tfrac{1}{|\nabla f|^2}\Ric(e_1,e_1)
    - \Ric_{\bar g_s}(e_1,e_1)\\
    &= (1+O(R))R
    \overline{R}_{11}
    + (1+O(R))
    \left( R_{1441} - (1+O(R))
    (R_{11}(R-R_{44}) - R_{11}^2 - o(1)R^2)  \right)\\
    &= R_{1441} +R(\overline{R}_{11}-R_{11})+ 
    R_{11}^2 + o(1)R^2\\
    &= \tfrac{1}{2}R^2 + o(1)R^2
    >0.
\end{align*}
\end{proof}

\subsection{Estimates for The Warping Function and Its Derivatives}

In this section we prove apriori estimates for the warping function $F$ and its derivatives. These will play an important role in deriving asymptotics for the soliton in the cylindrical region in the next sections. Unless otherwise stated, we are restricting everything to $\Sigma$ as discussed before. However, as mentioned before, we will oocasionally use the notation $\Sigma_s$ to refer to $\Sigma$ equipped with $\bar g_s$.

First of all note that by standard computations for warped products, we have
\[
    \Ric_{\bar g_s}
    = - 2\frac{F_{zz}}{F} \rd z^2
    + \left( -\frac{F_{zz}}{F} + \frac{1-F_z^2}{F^2}\right)F^2 g_{\IS^2}
\]

Using the evolution equation \eqref{oureq} for $\bar{g}_s$, we compute the evolution equation for the warping function $F(z,s)$ in the Proposition below. 

\begin{Proposition}
In commuting variables $z$ and $s$, we have
     \begin{align*}
     -F_s
     &= F_{zz} - F^{-1}(1+F_z^2)
     + 2F_z \left\{\frac{F_z}{F}(0,s)
      - \int_0^z \frac{F_z^2}{F^2}(\zeta,s)\, \rd \zeta\right\}\\
    &\qquad 
    - F^{-1}\bEE^{\rm orb}
    + F_z\int_0^z \bEE^{\rm rad}(\zeta,s)\,\rd \zeta.
 \end{align*}
\end{Proposition}
\begin{proof}
Fix $p\in \Sigma$ and let $B$ be the $O(3)$-orbit passing through $p$. Recall that $z(p,s)$ is the signed distance form our basepoint $q$ to $B$ measured with respect to $\bar g_s$. On one hand,
\[
   \left.\partial_s\right\rvert_{p=const.} F
   = \left.\partial_s\right\rvert_{z=const.} F+ z_s F_z,
\]
and by standard distance distortion,
 \[
    z_s=\partial_s z(p,s)= \int_0^{z}
    \tfrac{1}{2}(\partial_s \bar g_s)(\partial_z,\partial_z)\, ds
    = - 2\int_0^{z}
    \frac{F_{zz}}{F}(\zeta,s)\, \rd \zeta
    + \int_0^{z(s)}\bEE^{\rm rad}(\zeta,s)\, \rd \zeta.
 \]
On the other hand,
 \[
    \left.\partial_s\right\rvert_{p=const.} 
    (F^2\, g_{\IS^2})
    = 2(- F_{zz}F +1-F_z^2 + \bEE^{\rm orb})g_{\IS^2},
\]
and thus
\[
    \left.\partial_s\right\rvert_{p=const.}F
    = - F_{zz} + F^{-1}(1-F_z^2)
    + F^{-1}\bEE^{\rm orb}
\]
Hence, in commuting variables $z$ and $s$ we have
 \begin{align*}
     -F_s
     &= F_{zz} - F^{-1}(1-F_z^2)
      - 2F_z\int_0^z \frac{F_{zz}}{F}(\zeta,s) \, \rd \zeta\\
    &\qquad
    -  F^{-1}\bEE^{\rm orb}
    + F_z \int_0^z \bEE^{\rm rad}(\zeta,s)\,\rd \zeta
 \end{align*}
 The conclusion follows by integration by parts.
\end{proof}


We prove the following Lemma that will be implicitly used.
\begin{Lemma} \label{F-highder-est}
For $s\geq \underline{s}$,
\begin{gather*}
    |F_z|\le 1,\qquad
    F_{zz} < 0,\qquad
    F^{m}\partial_z^{m+1} F \leq C_m (1+F|F_{zz}|)^m
\end{gather*}
for every $m\geq 0$ and some constant $C_m$. 
\end{Lemma}
\begin{proof}
By Gauss formula, $\sec_{\Sigma_s}>\sec_{M}\ge 0,$ see \cite[Lemma 2.1]{CMZ23} for example.
Recall that the sectional curvature of $\Sigma_s$ in the radial direction is $-F_{zz}/F$, and the curvature in the orbital direction is $(1-F_z^2)/F^2$. So, $F_{zz}<0$ and $|F_z|\le 1.$
By \eqref{eq: gauss}, \eqref{eq: sec grad f} and Perelman's derivative estimates in Theorem \ref{thm: Rm der est}, $|\nabla_k (R-\overline{R})| = O(R^{2+\frac{k}{2}}).$ The last assertion follows verbatim as the three-dimensional case in \cite[Proposition 3.4]{Bre20}.
\end{proof}


We need the following rough estimate.

\begin{Lemma} Let $r_{\max}(s) := \sup_{z} F(z,s)$ denote the maximum radius of the symmetric spheres in $\Sigma_s$. 
\label{lem: r_max two sided}
After possibly shifting the {
initial time by $s_0$, we have
    \[
    2s
       \le r_{\max}^2(s+s_0) \le (1+o(1))2s,
    \]
    as $s\to\infty.$ }
\end{Lemma}
\begin{proof}
Since $F(0,s)=(1+o(1))\sqrt{2s}$ as $s\to\infty$, we may assume
\[
    r_{\max}(s) \ge F(0,s) \ge \sqrt{s}
\]
Let $z_0$ be the point that achieves the maximum radius on $\Sigma_s$. At $(z_0,s)$,
    \[
        F = r_{\max}(s),\quad
        F_z = 0,\quad
        F_{zz}\le 0.
    \]
So at $(z_0,s)$, the evolution equation of $F$ and the fact that $\bEE^{\rm orb}\ge 0$ (Lemma \ref{lem: Eorb>0}) imply
\begin{align*}
    - r_{\max}'(s)
    &= - F_s \le - F^{-1} - \bEE^{\rm orb} F^{-1}
    \le - r_{\max}^{-1},
\end{align*}
which can be integrated to
\[
        r_{\max}^2(s)
        \ge 2s - C.
\]
for large $s$. On the other hand, since $F(z_0,s) \geq \sqrt{s}$, we can use Proposition \ref{prop: neck lemma}. So for any given $\epsilon$, we can assume $z_0$ to be an $(\eps,2)$-center if $s\ge\underline{s}(\eps)$. This implies
    \[
        F_{zz}\ge -\eps F^{-1},
   \]
   and at $(z_0,s)$, Lemma \ref{lem: error by F} gives
   \[
        -F_s \ge -\eps F^{-1} - F^{-1} - C F^{-3}
        \ge -2\eps F^{-1} - F^{-1}.
   \]
   Thus, for $s\ge \underline{s}$ 
   \[
        r^2_{\max}(s) \le 2(1+2\eps) s.
   \]
  
    By considering $\bar g_{s_0+s}$ instead of $\bar g_s$ in the definition \eqref{eq: bar g_s def} for sufficiently large $s_0,$ we may assume $C=0$ while having the same upper bound.  More precisely, we have $r^2_{\max}(s+s_0) \ge 2s$. 
    

 
\end{proof}

We are now ready to show the bound on the first derivative of $F$.

\begin{Proposition}
\label{prop: Fz bdd by barrier}
There exist a universal number $\underline{a}$ and a large $D$ (depending only on $M^4$) such that the following holds: 
If  $a\ge \underline{a}$ is chosen large enough such that 
    \[ \frac{r_{\max}(s)}{\sqrt{2s}} \le 1 + \tfrac{1}{100}a^{-2}
    \]
    holds for $s\ge Da^3$, then
    \[
        F_z^2(z,s) \le \psi_a\left( \frac{F(z,s)}{\sqrt{2s}}\right)
    \]
    whenever $s\ge Da^3$ and $F(z,s)\ge r_*a^{-1}\sqrt{2s}$.
\end{Proposition}
\begin{proof}
    
The proof is similar to that of \cite[Proposition 3.8]{ABDS22}, except we have to deal with  extra error terms here. we will outline the details below. 

Take $\underline{a}$ large enough so that the barrier construction in \ref{prop: barrier spatial} works and $s\geq \underline{a}^3$ implies all the previous estimates, and fix $a\geq \underline{a}$ in the statement. Let $\mathcal{I}$ be the set of all  $s\ge Da^3$ such that the conclusion holds. Suppose, on the contrary, $\mathcal{I}$ is not all of $[Da^3,\infty)$.
By Proposition \ref{prop: neck lemma} and the lower bound for $\psi_a$ in \ref{prop: barrier spatial}, $\mathcal{I}\supseteq [C(a),\infty)$ holds for some constant $C(a)$ depending on $a$.
Let $\bar s = \sup \left([Da^3,\infty)\setminus \mathcal{I}\right)$. Then there is a point $(\bar z, \bar s)$ with 
\[
    F_z^2 = \psi_a\left(\frac{F}{\sqrt{2\bar s}}\right),\quad
    F(\bar z, \bar s) \ge r_* a^{-1}\sqrt{2\bar s}
\]
We can argue as in the proof of \cite[Proposition 3.8]{ABDS22} to see that in a neighborhood of $(\bar z, \bar s)$, we may find a smooth function $u$ such that
\[
    F_z^2(z,s) = u(F(z,s),s).
\]
The function $u$ satisfies the equation
\begin{align*}
    -u_s&= uu_{rr} - \tfrac{1}{2}u_r^2 + r^{-2}(1-u)(ru_r + 2u) \\
    &\qquad + 2u^2 \tilde{\mathcal{E}}^{\rm rad} + r^{-1}u_r \tilde{\mathcal{E}}^{\rm orb}
    - 2u (\tilde{\mathcal{E}}^{\rm orb}/r)_r
\end{align*}
We postpone deriving this equation to the end of the proof.

We shall carefully estimate the contribution from $\bEE$ in the equation for $u$. First, $u\leq 1$ (Lemma \ref{F-highder-est}) and $\bEE$ estimates in Lemma \ref{lem: error by F} imply
\begin{equation}
\label{ineq: Etilde}
u|\tilde{\mathcal{E}}^{\rm rad}| 
 \le C r^{-4},\quad
 |\tilde{\mathcal{E}}^{\rm orb}|
 \le C r^{-2},\quad
 \sqrt{u}|\partial_r \tilde{\mathcal{E}}^{\rm orb}|
 \le C r^{-3}.
\end{equation}
Now, \eqref{ineq: Etilde} along with $u\leq 1$, $u_r<0$ (Lemma \ref{F-highder-est}) and $\EE^{\rm orb}\ge 0$ (Lemma \ref{lem: pos Eorb}) imply
\[
    2u^2 \tilde{\mathcal{E}}^{\rm rad}
    \le C \sqrt{u} r^{-4},\qquad
    r^{-1}u_{r} \tilde{\mathcal{E}}^{\rm orb} <0.
\]
Using \eqref{ineq: Etilde} again we get
\begin{align*}
    - 2u (\tilde{\mathcal{E}}^{\rm orb}/r)_r
    &= 2u \tilde{\mathcal{E}}^{\rm orb} r^{-2}
    - 2u\partial_r\tilde{\mathcal{E}}^{\rm orb}/r
    \le  Cur^{-4} + C\sqrt{u}\, r^{-4}
    \le 2C\sqrt{u}\,r^{-4}.
\end{align*}
Thus, $u$ satisfies
\begin{equation}
\label{ineq: u evo}
    -u_s
    \le uu_{rr} - \tfrac{1}{2}u_r^2 + r^{-2}(1-u)(ru_r + 2u) 
    + C_0 \sqrt{u}\,r^{-4},
\end{equation}
for some constant $C_0$.

On the other hand, by Corollary \ref{cor: spacetime barrier}
$\Psi_a(r,s):=\psi_a(r/\sqrt{2s})$ satisfies
\begin{equation}
\label{ineq: Psi evo}
    -\Psi_{a,s}
        > \Psi_a \Psi_{a,rr} - \tfrac{1}{2}\Psi_{a,r}^2
        + r^{-2}(1-\Psi_a)(r\Psi_{a,r} + 2\Psi_a)
        + Q,
\end{equation}  
where
    \[
       Q\ge  \left\{ \begin{array}{ll}
              \tfrac{1}{8} a^{-4}s^{-1}(r/\sqrt{2s})^{-5}, & \text{ for } r/\sqrt{2s}\in [Na^{-1},1+\frac{1}{100}a^{-2}]\\
             \tfrac{1}{4}a s^{-1}, & \text{ for }r/\sqrt{2s} \in [r_*a^{-1},Na^{-1}]
         \end{array} \right.
    \] 
It is enough to show
\begin{equation}
\label{ineq: Q bounds errors}
    Q \ge 2C_0 \sqrt{u}\,r^{-4},
\end{equation}  
where $C_0$ is the same constant as in \eqref{ineq: u evo}. If \eqref{ineq: Q bounds errors} holds, we can use parabolic maximum principle as in Proposition 3.8 of \cite{ABDS22} and compare \eqref{ineq: u evo} and \eqref{ineq: Psi evo} to get a contradiction. This would imply that $\bar{s}$ as described above does not exist, hence concluding the proof of the Proposition.

Let $\bar r=F(\bar z,\bar s)$. We first consider the case $\bar r/\sqrt{2\bar s}\in [Na^{-1},1+\frac{1}{100}a^{-2}]$. On one hand, by $s\ge D a^{3}$ and Proposition \ref{prop: barrier spatial},

\begin{align*}
    & \bar r^4 \, Q  \geq \bar r^4 \, \big(\tfrac{1}{8} \bar s^{-1} a^{-4}(\bar r/\sqrt{2\bar s})^{-5}\big)
    = \tfrac{1}{2}
    \bar s a^{-4} (\bar r/\sqrt{2\bar s})^{-1} \geq \tfrac{1}{2}
    \bar D a^{-1} (\bar r/\sqrt{2\bar s})^{-1}  
\end{align*}
On the other hand, at $(\bar r,\bar s)$ we have 
\[
    u(\bar r,\bar s)
    = \psi_a(\bar r/\sqrt{2\bar s})
    \le  2a^{-2}(\bar r/\sqrt{2\bar s})^{-2}
\]
So we have $Q \geq 2C_0 \sqrt{u} r^{-4}$ if we choose $D>8C_0$. So \eqref{ineq: Q bounds errors} holds in this case.

We then consider the case where  $\bar r/\sqrt{2\bar s}\in (r_*a^{-1},Na^{-1}].$
We have
\[
    \tfrac{1}{4}a \bar s^{-1} \bar r^4 \ge \tfrac{1}{4} a\bar s^{-1} (r_*a^{-1}\sqrt{2\bar s})^4
    =  r_*^4 a^{1-4} \bar s
    \ge r_*^4D > 2C_0,
\]
if we choose $D$ large enough. So, $Q\ge 2C_0 r^{-4} \ge 2C_0 \sqrt{u}\,r^{-4}$ in this case as well.

Thus, \eqref{ineq: Q bounds errors} holds for both cases. 
This concludes the proof of the Proposition as discussed above.

Finally, we compute the equation for $u$. After the $s$-dependant change of variable $r=F(z,s)$, $\bar g_s$ can be written as
\[
    \tg = u(r,s)^{-1}\rd r^2 + r^2\, g_{\IS^2}
\]
and for some $V=v(r,s)\partial_r$
\[
    \partial_s \tg = 2\Ric_{\tg} + 2 \bEE  + \LL_V \tg
\]

Similar to \cite{Bre20}, we have
\begin{gather*}
    \LL_V \tg = \left(-u^{-2}u_rv + 2u^{-1}v_r\right) \rd r^2
    +2rv\, g_{\IS^2}\\
    \Ric_{\tg} = - \frac{u_r}{ru} \rd r^2 + \left(1-u-\tfrac{1}{2}ru_r\right) g_{\IS^2}
\end{gather*}
and we can write
\[
     \tilde{\mathcal{E}}=: \tilde{\mathcal{E}}^{\rm rad} \, \rd r^2 + \tilde{\mathcal{E}}^{\rm orb}\, g_{\IS^2}
\]

Since
\[
    \partial_s \tg = - u^{-2}u_s \, \rd r^2
\]
comparing the spherical component gives
\[
    v = \frac{u-1}{r} + \frac{u_r}{2} - \frac{\tilde{\mathcal{E}}^{\rm orb}}{r}
\]
and in turn comparing the radial component gives
\begin{align*}
    -u^{-2}u_s &= - \frac{2u_r}{ru} + 2\tilde{\mathcal{E}}^{\rm rad} 
    - u^{-2} u_r \left(\frac{u-1}{r} + \frac{u_r}{2} - \frac{\tilde{\mathcal{E}}^{\rm orb}}{r}\right)
    +2 u^{-1} \left(\frac{u-1}{r} + \frac{u_r}{2} - \frac{\tilde{\mathcal{E}}^{\rm orb}}{r}\right)_r
\end{align*}
After simplification, we get
\begin{align*}
    -u_s&= uu_{rr} - \tfrac{1}{2}u_r^2 + r^{-2}(1-u)(ru_r + 2u) \\
    &\qquad + 2u^2 \tilde{\mathcal{E}}^{\rm rad} + r^{-1}u_r \tilde{\mathcal{E}}^{\rm orb}
    - 2u (\tilde{\mathcal{E}}^{\rm orb}/r)_r
\end{align*}

\end{proof}

Let us now perform parabolic rescaling
\[
    G(\xi,\tau)=e^{\frac{\tau}{2}}F\left(e^{-\frac{\tau}{2}}\xi,e^{-\tau}\right)-\sqrt{2}.
\]
with the following change of variable:
\[
    z = e^{-\frac{\tau}{2}}\xi,\quad
    s = e^{-\tau}.
\]
Note that
\[
    G_{\xi} = F_z,\quad
    G_{\xi\xi} = e^{-\frac{\tau}{2}}F_{zz}.
\]

Accordingly, we define
\begin{align*}
    \EE^{\rm rad}(\xi,\tau) &:= e^{-\tau}\bEE^{\rm rad}\left(e^{-\frac{\tau}{2}}\xi, e^{-\tau}\right),\\
    \EE^{\rm orb}(\xi,\tau) &:=\bEE^{\rm orb}\left(e^{-\frac{\tau}{2}}\xi, e^{-\tau}\right),
\end{align*}
so that
\[
    \EE^{\rm rad} d\xi^2 + \EE^{\rm orb} g_{\IS^2}
    = \bEE^{\rm rad} \rd z^2 + \bEE^{\rm orb} g_{\IS^2}
    = \bEE\left(e^{-\frac{\tau}{2}}\xi, e^{-\tau}\right).
\]

By convergence to the cylinder (see also Lemma \ref{lem: r_max two sided}) we have
\[
    G(\cdot,\tau)\to 0,
\]
as $\tau\to -\infty,$ locally smoothly.

We compute the evolution equation for $G$.
\begin{Proposition}
\label{prop: G eqn}
The function $G$ satisfies
    \begin{align*}
        G_{\tau}&= G_{\xi\xi}- \tfrac{\xi}{2} G_{\xi} + \tfrac{1}{2}(\sqrt{2}+G)
        - (\sqrt{2}+G)^{-1}(1+G_\xi^2)
        + 2G_\xi \left\{ 
        \frac{G_{\xi}}{\sqrt{2}+G}(0,\tau)
        - \int_0^{\xi}
        \frac{G_{\xi}^2(x,\tau)}{(\sqrt{2}+G)^2}\, \rd x\right\}\\
        &\qquad 
        - (\sqrt{2}+G)^{-1}\EE^{\rm orb}
        + G_{\xi} \int_0^{\xi}\EE^{\rm rad}(x,\tau)\, \rd x.
    \end{align*}
\end{Proposition}
\begin{proof}

Then
    \begin{align*}
        G_{\tau}
        &= \tfrac{1}{2}(\sqrt{2}+G) - \tfrac{1}{2}\xi F_z - e^{-\frac{\tau}{2}}F_s\\
        &= \tfrac{1}{2}(\sqrt{2}+G) - \tfrac{\xi}{2} G_{\xi}
        + e^{-\frac{\tau}{2}}F_{zz}
        - e^{-\frac{\tau}{2}} F^{-1}(1+F_z^2)
        + 2e^{-\frac{\tau}{2}}F_z \left\{\frac{F_z}{F}(0,s)
      - \int_0^z \frac{F_z^2}{F^2}(\zeta,s)\, \rd \zeta\right\}\\
        &\qquad
        - e^{-\frac{\tau}{2}} F^{-1}\bEE^{\rm orb}
        + e^{-\frac{\tau}{2}} F_z\int_0^z \bEE^{\rm rad}(\zeta,s)\,\rd \zeta\\
        &= G_{\xi\xi}- \tfrac{\xi}{2} G_{\xi} + \tfrac{1}{2}(\sqrt{2}+G)
        - (\sqrt{2}+G)^{-1}(1+G_\xi^2)
        + 2 G_\xi \left\{  
        \frac{G_{\xi}}{\sqrt{2}+G}(0,\tau)
        - \int_0^{\xi}
        \frac{G_{\xi}^2}{(\sqrt{2}+G)^2}(x,\tau)\, \rd x\right\}\\
        &\qquad 
        - (\sqrt{2}+G)^{-1}\EE^{\rm orb}
        + G_{\xi} \int_0^{\xi}\EE^{\rm rad}(x,\tau)\, \rd x.
    \end{align*}
\end{proof}

\begin{Definition}
    For each $\tau$, define
    \[
        \rho_{\max}(\tau) := e^{\frac{1}{4}\tau} + \sup_{\xi} G(\xi,\tau) 
        = e^{\frac{\tau}{2}}r_{\max}(e^{-\tau})-\sqrt{2}+ e^{\frac{1}{4}\tau}
    \]
    and
    \[
        \rho(\tau)
        := \sup_{\sigma\le \tau} \, \rho_{\max}(\sigma),\qquad
        \delta(\tau) :=  \rho(\tau) + \sup_{\sigma\le \tau} |G(0,\tau)|.
    \]
\end{Definition}
Note that by Lemma \ref{lem: r_max two sided}, for $\tau\le\bar\tau$ we have
\[
    e^{\frac{\tau}{2}}r_{\max}(e^{-\tau})\ge e^{\frac{\tau}{2}}\sqrt{2e^{-\tau}} = \sqrt{2}
\]
and therefore
\[
\rho_{\max}(\tau)\ge e^{\frac{1}{4}\tau}>0,\qquad
    \delta(\tau)\ge \rho(\tau)\ge e^{\frac{1}{4}\tau}>0.
\]
if $-\bar{\tau}$ is large. On the other hand, convergence to a cylinder implies
\[\rho(\tau) = o(1),\,\,\,\, \delta(\tau)=o(1) \qquad  \mbox{as} \qquad \tau\to -\infty.\]

By the definitions above and Proposition \ref{prop: Fz bdd by barrier} we have the following.

\begin{Proposition}
\label{prop: G_xi bdd by barrier}
    There is $-\bar \tau$ large enough such that
    \[
        G_{\xi}^2(\xi,\tau)
        \le \psi_a \left(1+\frac{G(\xi,\tau)}{\sqrt{2}}\right)
        \le C \rho(\tau),
    \]
    whenever $\tau\le \bar\tau$ and $\sqrt{2}+G(\xi,\tau)\ge r_*a^{-1}\sqrt{2}$, where $a = 
    10^{-1}\rho^{-\frac{1}{2}}(\tau)$.
\end{Proposition}
\begin{proof}
    We perform the change of variable $s=e^{-\tau}, z=e^{-\frac{\tau}{2}}\xi$. Fix $-\bar\tau$ large enough
    so that $\rho(\tau)\ge e^{\frac{1}{4}\tau}$ for $\tau\le \bar\tau$, and take an arbitrary $\tau_* \leq \bar \tau$. Taking $a = 
    10^{-1}\rho^{-\frac{1}{2}}(\tau_*)$, we can assume $a \geq \underline a$ since $\rho(\tau) \to 0$. For $\tau \leq \tau_*$ 
    \[
        s=e^{-\tau}\ge e^{- \tau_*}
        \ge D e^{-\frac{3}{8}\tau_*}
        \ge D \rho^{-\frac{3}{2}}(\tau_*)
        \ge D a^{3}
    \]
    if $-\bar\tau$ is large enough, and
    \[
        1+\frac{G(\cdot,\tau)}{\sqrt{2}}
        \le 1 + \rho_{\max}(\tau)
        \le 1 + \frac{1}{100}a^{-2}.
    \]
    The first inequality follows from Proposition \ref{prop: Fz bdd by barrier} for $\tau \leq \tau_*$ and $\sqrt{2}+G(\xi,\tau)\ge 10 r_*\sqrt{2} \rho(\tau_*)^\frac{1}{2}$. By the properties of the barrier function $\psi_a$ (see Proposition \ref{prop: barrier spatial}), we have $\psi_a \leq Ca^{-2} = C \rho(\tau_*)$. We get both estimates by choosing $\tau=\tau_*$, and since $\tau_* \leq \bar\tau$ was arbitrary, we are done.
\end{proof}

\begin{Lemma}
\label{lem: a priori bd on G}
If $\tau\le\bar\tau$, $|\xi|\le \rho^{-\frac{1}{100}}(\tau)$ and $a=10^{-1}\rho(\tau)^\frac{-1}{2}$, then
\[
    G(\xi,\tau)\ge -1> (r_*a^{-1}-1)\sqrt{2}.
\]
\end{Lemma}
\begin{proof}
The second inequality follows from $\rho(\tau) \to 0$. As $\tau \to -\infty$, we have $G(\xi,\tau) \to 0$ uniformly on compact sets in $\xi$. So we can assume $G(0,\tau)>\frac{-1}{4}$ near the origin. We argue by contradiction. Assume there is a point with $|\xi| \le \rho^{-\frac{1}{100}}(\tau)$ such that $G(\xi,\tau) < -1$. Then we can find a point $\bar\xi$ closest to the origin satisfying $|\bar\xi| \le \rho^{-\frac{1}{100}}(\tau)$ and $G(\bar \xi,\tau)=-1$. Without loss of generality we may assume $\bar{\xi} > 0$. For $\xi\in [0,\bar\xi)$, we have $G(\xi,\tau)>-1> (r_*a^{-1}-1)\sqrt{2}$, and by Proposition \ref{prop: G_xi bdd by barrier} we get
\[
    G_{\xi}^2(\xi,\tau)
    \le C \rho(\tau).
\]

Since $\delta(\tau) \to 0$ and $\rho(\tau) \to 0$ as $\tau\to-\infty$, it follows that
    \[
        G(\bar\xi,\tau)
        =G(0,\tau)
        + \int_0^{\bar\xi} G_{\xi}(x,\tau)\, \rd x
        \ge -\delta(\tau)
        - C \rho^{\frac{1}{2}}(\tau) \rho^{-\frac{1}{100}}(\tau)
        \ge -1/2,
    \] which is a contradiction.  
\end{proof}

\begin{Lemma}
\label{lem: C0 C1 on G}
For $-\bar{\tau}$ sufficiently large, for $\tau\le \bar\tau$ and $|\xi|\le \delta^{-\frac{1}{100}}(\tau)$, we have
\[
    \left|G(\xi,\tau)\right|+|G_{\xi}(\xi,\tau)|
    \le C\delta^{\frac{1}{4}}(\tau).
\]
\end{Lemma}
\begin{proof}
    Note first that $|G_{\xi}|^2\le C\delta(\tau)$ follows from Lemma \ref{lem: a priori bd on G} and Proposition \ref{prop: G_xi bdd by barrier}.
    By the definition of $\delta(\tau)$ we have $|G(0,\tau)|\le \delta(\tau)$. So
    $\left|G(\xi,\tau)\right|\le C\delta^{\frac{1}{4}}(\tau)$ follows by integrating the bound on $|G_{\xi}|$.
\end{proof}

For simplicity, in the following we write
\[
    I_{\tau}:=\left\{\xi: |\xi|\le \delta^{-\frac{1}{100}}(\tau)\right\},\qquad
    \rd\nu(\xi) = (4\pi)^{-\frac{1}{2}}e^{-\frac{\xi^2}{4}}\, d\xi.
\]
and we assume $-\bar\tau$ is large enough so that the previous lemmas hold. 
We are now ready to prove the higher order derivative estimates for $G(\xi,\tau)$. Given Theorem \ref{perelmanbds}, the proofs are similar to the three dimensional case.
\begin{Lemma}
    For $\tau\le \bar\tau$ and \, $\xi\in I_\tau$, we have
\[
    |G_{\xi\xi}(\xi,\tau)|
    \le C\delta^{\frac{1}{8}}(\tau).
\]
\end{Lemma}
\begin{proof}
    The proof is similar to \cite[Lemma 3.8]{Bre20}, \cite[Lemma 3.14]{ABDS22}.
\end{proof}

\begin{Lemma}
    For $\tau\le \bar\tau$, and  \,$\xi\in I_\tau$,
\[
    |\partial_\xi^{m}G(\xi,\tau)|
    \le C(m),
\]
for each $m\ge 0.$
\end{Lemma}
\begin{proof}
    The proof is similar to \cite[Lemma 3.9]{Bre20}, \cite[Lemma 3.15]{ABDS22}.
\end{proof}

\begin{Lemma}
\label{lem: L4 of G_xi}
    For $\tau\le\bar\tau,$
    \[
        \left| G_{\xi}(0,\tau)\right|^4
        \le C \delta^{\frac{1}{100}}(\tau)
        \int_{I_\tau}  \left| G(\xi,\tau)\right|^2\, \rd\nu(\xi),
    \]
    and
    \[
        \int_{I_\tau} 
        \left|  G_{\xi}(\xi,\tau)\right|^4\, \rd\nu(\xi)
        \le  C \delta^{\frac{1}{100}}(\tau)
        \int_{I_\tau}  \left| G(\xi,\tau)\right|^2\, \rd\nu(\xi)
        + C\exp\left(-\tfrac{1}{8}\delta^{-\frac{1}{50}}(\tau)\right).
    \]
\end{Lemma}
\begin{proof}
    The proof is similar to c.f. \cite[Lemma 3.10]{Bre20}, \cite[Lemma 3.16]{ABDS22}.
\end{proof}

\begin{Lemma}
\label{lem: error by tau}
For $\tau \le \bar \tau$ and any $ k \ge 0$,
\[
    \left|\EE^{\rm rad}\right| \le C e^{\tau} (\sqrt{2}+G)^{-4},\qquad
    \left|\partial_\xi^k\EE^{\rm orb}\right| \le C_k e^{\tau} (\sqrt{2}+G)^{-2}.
\]
\end{Lemma}
\begin{proof}
    This follows directly by  rescaling the error terms and Lemma \ref{lem: error by F}.
\end{proof}

Putting the above estimates together, we get the following Lemma.

\begin{Lemma}
\label{lem: bound on parabolic op}
    For $\tau\le\bar\tau,$
    \begin{align*}
        \int_{I_\tau} 
        \Big|  G_\tau - G_{\xi\xi} + \frac{\xi}{2}G_\xi-G\Big|^2&(\xi,\tau)\, d\nu(\xi)\le  \\
        & C \delta^{\frac{1}{100}}(\tau)
        \int_{I_\tau}  \left|G(\xi,\tau)\right|^2\, d\nu(\xi)
        + C\exp\left(-\tfrac{1}{8}\delta^{-\frac{1}{50}}(\tau)\right)
        + C e^{2\tau}.
    \end{align*}
        
\end{Lemma}
\begin{proof}
The proof is similar to the three-dimensional case, c.f. \cite[Lemma 3.11]{Bre20}, \cite[Lemma 3.17]{ABDS22}. We only need to estimate the extra error terms.
For $\xi\in I_{\tau}$ Lemma \ref{lem: error by tau} implies
\[
    \left|(\sqrt{2}+G)^{-1}\EE^{\rm orb}\right|
    \le C e^{\tau},
\]
and Lemma \ref{lem: C0 C1 on G} and Lemma \ref{lem: error by tau} imply
\begin{align*}
\left|G_{\xi}\int_0^\xi \EE^{\rm rad}(x,\tau)\, \rd x\right|   
&\le 
|G_{\xi}| e^{\tau} \delta^{-\frac{1}{100}}(\tau)
\le C e^{\tau}.
\end{align*}

\end{proof}

\def \HH {\mathcal{H}}




\subsection{Dichotomy in asymptotic behavior}

As in \cite{Bre20,ABDS22}, we consider the operator
\[
    \LL G := G_{\xi\xi}
    - \tfrac{1}{2}\xi G_{\xi}
    + G
\]
on the weighted $L^2$ Hilbert space $\HH$ with the inner product
\[
    \langle u, v \rangle_{\HH}
    := \int_{\IR} uv\, d\nu.
\]
and the norm $\|G\|_{\HH} = \langle G,G \rangle_{\HH}^{\frac{1}{2}}$.

Recall that $\LL$ admits a spectral decomposition. 
The eigenvalues of $\LL$ are $1-\frac{n}{2}$ for $n\geq 1$, and the corresponding eigenfunctions are $H_n(\xi/2)$, where $H_n$ is the $n^{\rm th}$ Hermite polynomial. We let
\[
    h_0(\xi):= 1,\quad
    h_{1}(\xi)
    := \tfrac{1}{\sqrt{2}}\xi,\quad
    h_2(\xi)
    := \tfrac{1}{2\sqrt{2}}(\xi^2-2),\quad
    h_{3}(\xi)
    := \tfrac{1}{4\sqrt{3}}(\xi^3-6\xi),\quad \cdots
\]
be the unit eigenfunctions of $\LL$ corresponding to eigenvalues $1,\frac{1}{2},0,-\frac{1}{2},\cdots$. We also let $\HH = \HH_+ \oplus \HH_0 \oplus \HH_-$ to be the decomposition of $\HH$ into subspaces generated by $h_n$ similar to \cite{Bre20}, and $P_+, P_0, P_-$ to be the orthogonal projections unto those subspaces.

\def \HG {\widehat{G}}

Let $\eta$ be a smooth function on $\IR$ such that $\eta = 1$ over $[-\frac{1}{2}, \frac{1}{2}]$, $\eta(s)=0$ for $s\in \IR\setminus [-1,1]$, and $s\eta'(s)\le 0$ for all $s\in \IR$. Now let
\[
    \phi(\xi,\tau)
    := \eta\left(\delta^{\frac{1}{100}}(\tau)\xi\right),\quad
    \HG(\xi,\tau) := \phi(\xi,\tau)G(\xi,\tau)
\]
and define
\begin{align*}
    \gamma(\tau) &:= \int_{\IR} |\HG|^2(\xi,\tau)\, d\nu(\xi),\\
    \gamma^+(\tau) &:= \int_{\IR} |P_+\HG|^2(\xi,\tau)\, d\nu(\xi),\\
    \gamma^0(\tau) &:= \int_{\IR} |P_0\HG|^2(\xi,\tau)\, d\nu(\xi),\\
    \gamma^-(\tau) &:= \int_{\IR} |P_-\HG|^2(\xi,\tau)\, d\nu(\xi).
\end{align*}

\begin{Lemma}\label{gamma-sys}
The quantities $\gamma$, $\gamma^+$ etc. defined above satisfy 
\begin{align*}
    \gamma^+(\tau-1)
    &\le e^{-1}\gamma^+(\tau)
    + C\delta^{\frac{1}{200}}(\tau)
    \sup_{[\tau-1,\tau]}\gamma
    + C\exp\left(-\frac{1}{64}\delta^{-\frac{1}{50}}(\tau)\right)+Ce^{2\tau}\\
\left|\gamma^0(\tau-1)-\gamma^0(\tau)\right|
    &\le \delta^{\frac{1}{200}}(\tau)
    \sup_{[\tau-1,\tau]}\gamma
    + C\exp\left(-\frac{1}{64}\delta^{-\frac{1}{50}}(\tau)\right)+Ce^{2\tau}\\
    \gamma^-(\tau-1)
    &\ge e\,\gamma^-(\tau)
    - C\delta^{\frac{1}{200}}(\tau)
    \sup_{[\tau-1,\tau]}\gamma
    - C\exp\left(-\frac{1}{64}\delta^{-\frac{1}{50}}(\tau)\right)-Ce^{2\tau}
\end{align*}

\end{Lemma}
\begin{proof}
    the proof is similar to \cite[Lemma 3.18]{ABDS22} and uses Lemma\ref{lem: bound on parabolic op}.
\end{proof}

We now analyze the evolution of the function $\rho_{\max}(\tau)$. Recall that by Lemma \ref{lem: r_max two sided}, for $\tau\le \bar\tau$ we have
    \[
        \rho_{\max} (\tau) = 
        e^{\frac{1}{4}\tau}+ \sup_{\xi} G(\xi,\tau)
        \ge 
         e^{\frac{1}{4}\tau}.
    \]
Let
    $\xi^*_\tau$
denote the unique point in space where the function $G(\cdot,\tau)$ attains its maximum.

\begin{Lemma}
    For $\tau\le \bar\tau$,
    \[
        0 < -G_{\xi\xi}(\xi^*_\tau,\tau)
        \le C\gamma^{\frac{1}{4}}(\tau).
    \]
\end{Lemma}
\begin{proof}
    The proof follows verbatim as in the three-dimensional case, c.f. \cite[Lemma 3.19]{ABDS22} once we replace their Proposition 2.8 with our Proposition \ref{prop: neck lemma}.
\end{proof}

\begin{Lemma}
For $\tau\le \bar\tau$,
    \[
         \tfrac{{\rm d}}{{\rm d}\tau}
        \rho_{\max}(\tau)
        \ge  \tfrac{3}{4} \rho_{\max} - C \rho_{\max}^2 - C\gamma^{\frac{1}{4}}(\tau).
    \]
\end{Lemma}
\begin{proof}
    First, note that at $\xi^*_\tau$ we have $\sqrt{2}+G \geq C>0$ since the same is true near the origin. Hence $\tfrac{1}{2}(\sqrt{2}+G)
        - (\sqrt{2}+G)^{-1} \geq G -C G^2$ at the maximum. By the evolution equation of $G$, at $(\xi^*_\tau,\tau)$ we have
    \begin{align*}
        \tfrac{{\rm d}}{{\rm d}\tau}
        \rho_{\max}(\tau)
        &\ge 
        \tfrac{1}{4}e^{\frac{1}{4}\tau}+
        G_{\xi\xi}
        + \tfrac{1}{2}(\sqrt{2}+G)
        - (\sqrt{2}+G)^{-1} - (\sqrt{2}+G)^{-1}\EE^{\rm orb}\\
        &\ge \tfrac{1}{4}e^{\frac{1}{4}\tau}
        + \rho_{\max}- 
        e^{\frac{1}{4}\tau} - C (\rho_{\max}-e^{\frac{1}{4}\tau})^2
        -C\gamma^{\frac{1}{4}}
        - C e^{\tau}\\
        &\ge \tfrac{7}{8} \rho_{\max} - C \rho_{\max}^2 - C\gamma^{\frac{1}{4}}
        - \tfrac{1}{2} e^{\frac{9}{10}\tau}\\
        &\ge \tfrac{6}{8}\rho_{\max}
        - C\rho_{\max}^2
        - C\gamma^{\frac{1}{4}},
    \end{align*}
    where we have absorbed all the exponential terms into $e^{\frac{1}{4}\tau}$ and used the fact that $e^{\frac{1}{4}\tau}
    \le \rho_{\max}(\tau)$ in the last line. 
\end{proof}

\begin{Lemma}
For $\tau\le\bar\tau$,
    \[
        \rho_{\max}(\tau-1)
        \le e^{-\frac{1}{2}}
        \rho_{\max}(\tau)
        + C\sup_{[\tau-1,\tau]}\gamma^{\frac{1}{4}}.
    \]
\end{Lemma}
\begin{proof}
    By the previous Lemma,
    for $\tau\le \bar\tau$ we have
    \[
         \tfrac{{\rm d}}{{\rm d}\tau}
        \rho_{\max}(\tau)
        \ge  \tfrac{1}{2} \rho_{\max}(\tau) - C\gamma^{\frac{1}{4}}(\tau).
    \]
    The conclusion follows by integration since $\rho_{\max} \to 0$.
\end{proof}

\begin{Lemma}
    For $\tau\le\bar\tau,$
\begin{align*}
    &\gamma^+(\tau-1)
    +\rho_{\max}^{8-\frac{1}{200}}(\tau-1)
    \\
    \le &\  e^{-1}
    \left(\gamma^+(\tau)
    +\rho_{\max}^{8-\frac{1}{200}}(\tau)\right)
    + C\delta^{\frac{1}{200}}(\tau)
    \sup_{[\tau-1,\tau]}\gamma
    + C\exp\left(-\frac{1}{64}\delta^{-\frac{1}{50}}(\tau)\right)+C_1e^{2\tau}.
\end{align*}
\end{Lemma}

\begin{proof}
The proof follows from Lemma \ref{lem: bound on parabolic op} as in \cite{ABDS22}.
\end{proof}

Define
\begin{align*}
    \Gamma(\bar\tau)
    &:= \sup_{\tau\le\bar\tau}
    \Big(\gamma(\tau)
    +\rho_{\max}^{8-\frac{1}{200}}(\tau)\Big),&
    \Gamma^+(\bar\tau)
    &:= \sup_{\tau\le\bar\tau}
    \Big(\gamma^+(\tau)
    +\rho_{\max}^{8-\frac{1}{200}}(\tau)\Big),\\
     \Gamma^0(\bar\tau)
    &:= \sup_{\tau\le\bar\tau}
    \gamma^0(\tau),
    &
    \Gamma^-(\bar\tau)
    &:= \sup_{\tau\le\bar\tau}
    \gamma^-(\tau).
\end{align*}
\begin{Lemma}
\label{lem: ode for modes}
    For $\tau\le \bar\tau,$
\begin{align*}
    \Gamma^+(\tau-1)
    &\le e^{-1}\Gamma^+(\tau)
    + C\delta^{\frac{1}{200}}(\tau)\Gamma(\tau),\\
\left|\Gamma^0(\tau-1)-\Gamma^0(\tau)\right|
    &\le C\delta^{\frac{1}{200}}(\tau)\Gamma(\tau),\\
    \Gamma^-(\tau-1)
    &\ge e\,\Gamma^-(\tau)
    - C\delta^{\frac{1}{200}}(\tau)\Gamma(\tau).
\end{align*}
\end{Lemma}
\begin{proof}
    By the standard interpolation inequalities,
    $|G(0,\tau)|
    \le C\gamma^{\frac{1}{4}}(\tau)$, and thus $$\sup_{\sigma\le\tau}|G(0,\sigma)|
    \le C\Gamma^{\frac{1}{4}}(\tau).$$

On one hand, $e^{\frac{1}{4}\tau} \leq \rho(\tau) \leq \delta(\tau)$, so $e^{2\tau} \leq \delta^8(\tau)$. On the other hand, $\rho_{\max}^{8-\frac{1}{200}}(\sigma) \leq \Gamma(\tau)$ for $\sigma\leq \tau$ and hence $\rho(\tau)^{8-\frac{1}{200}} \leq \Gamma(\tau)$. So we have
\[
    \delta(\tau)^{8-\frac{1}{200}} \leq
    C \rho(\tau)^{8-\frac{1}{200}}+C\sup_{\sigma\le \tau} \left| G(0,\sigma)\right|^{8-\frac{1}{200}}
    \le C\Gamma(\tau).
\]
So, for $\tau\le\bar\tau,$
\[
    C_1e^{2\tau}+ \exp\left(-\frac{1}{64}\delta^{-\frac{1}{50}}(\tau)\right)
    \le C\delta^8(\tau)
    \le C\delta(\tau)^{\frac{1}{200}}\Gamma(\tau).
\]
The rest of the proof follows verbatim as in the three-dimensional case, c.f. \cite[Page 16]{ABDS22}.
\end{proof}

We record the following corollary of the proof of Lemma \ref{lem: ode for modes} for easier reference.

\begin{Corollary} \label{E by Gamma}
    In the equation $\partial_\tau \hat{G} = \mathcal{L} \hat{G} + E$, the source term $E$ satisfies $\| E \| \leq C\delta^\frac{1}{200}\Gamma$.
\end{Corollary}
\begin{proof}
    Consider the proof of the previous Lemma and the right hand side expression in \ref{lem: bound on parabolic op}.
\end{proof}
    
Now that we have Lemma \ref{lem: ode for modes}, we get the following Lemma inspired by Merle and Zaag, similar to the three-dimensional case.
\begin{Proposition}[{\cite[Proposition 3.23]{ABDS22}}]
\label{prop-merlezaag}
    We have either
    \[
        \Gamma^+(\tau)+\Gamma^-(\tau)
        \le o(1)\Gamma^0(\tau) \qquad\text{or}
        \qquad \Gamma^0(\tau)+\Gamma^-(\tau)
        \le o(1)\Gamma^+(\tau)
    \]
    as $\tau\to -\infty.$
\end{Proposition}

\section{Ruling out the case when the positive mode dominates}
\label{sec-rulling-out}
Throughout this section, in addition to Assumptions {\bf (A1)-(A3)} we also assume
\begin{equation}
\label{eq-A4}
\sup_{\partial B_r(o)} R = O(r^{-\eta}),
\end{equation}
for some $\eta\in (\frac{2}{3},1)$. We show that in this case, the neutral term $\Gamma^0$ has to dominate as $\tau\to -\infty$. We argue by contradiction. In view of the dichotomy in Proposition \ref{prop-merlezaag}, we may assume that the positive term dominates the other two, i.e.,
\[
    \Gamma^0(\tau) + \Gamma^-(\tau)
    = o(1)\, \Gamma^+(\tau),
\]
as $\tau\to -\infty.$
By Lemma \ref{lem: ode for modes},
\[
    \Gamma^+(\tau-1) \le e^{-1}\Gamma^{+}(\tau)
    + C \delta^{\frac{1}{200}}(\tau) \Gamma^+(\tau).
\]
Iterating this yields
\[
    \Gamma^+(\tau) \le O\left(e^{\frac{\tau}{2}}\right).
\]
and by definition of $\Gamma^+$, we get
\[
    \sup_\xi G(\xi, \tau) \leq \rho_{\max}(\tau)\le O\left(e^{\frac{\tau}{16}}\right)
\]

We try to follow the arguments in \cite[Section 4]{ABDS22}, but we need to modify them due to extra error terms in our case.

\begin{Definition}
    Given $0<\alpha<1$, we say that condition $(\star_\alpha)$ holds if
    \begin{equation}
    \label{ineq: star}
        \tag{$\star_\alpha$}
        r_{\max}^2(s) \le 2s\left(1 + O(s^{-\alpha})\right).
    \end{equation} 
\end{Definition}

As mentioned above, the condition \eqref{ineq: star} holds for  $\alpha = \frac{1}{16}$. 
We use an iteration argument similar to \cite{ABDS22} to show that \eqref{ineq: star} holds for $\alpha =2/3$.
this would imply a diameter bound for the level set that contradicts \eqref{eq-A4}.

\begin{Proposition}
\label{prop: F_z gamma alpha}
    Suppose \eqref{ineq: star} holds for some $ \tfrac{1}{16}<\alpha<1$. Choose $1 \leq \gamma \leq 2$ such that 
    \[
        1/\gamma \ge \max\left(1-\alpha, \tfrac{3}{2}\alpha\right).
    \]
    Then
    \[
        F_z^2(z,s) \le C s^{-\gamma\alpha},
    \]
    holds whenever $s\ge \underline{s}(\alpha)$ and and $F(z,s)\ge \sqrt{s}$. 
\end{Proposition}
\begin{proof}
We first prove the following Claim.

\noindent \textbf{Claim:}
    There is a large number $K$ and a small number $\theta$ with the following property: For any $a\ge K$ and any $1\leq \gamma \leq 2$ satisfying $\frac{1}{\gamma}\ge 1-\alpha$, $s\ge K^2 a^{\frac{2}{\gamma \alpha}}$ implies  
    \[
        1-\theta\le \frac{r_{\max}(s)}{\sqrt{2s + K a^{\frac{2}{\gamma\alpha}}}} \le 1+\tfrac{1}{100}a^{-2},
    \]
    
\begin{proof}[Proof of the Claim.]
    Let $\underline{s}$ be large such that the results of previous sections and $r^2_{\max}(s)\ge 2s$ (c.f. Lemma \ref{lem: r_max two sided}) hold for $s \geq \underline{s}$. Choose $K$ large enough so that $a \geq K, \alpha<1, \gamma \leq 2$ and $s\ge K^2 a^{\frac{2}{\gamma \alpha}}$ imply $s \geq \underline{s}$.

    For $s\ge K^2 a^{\frac{2}{\gamma \alpha}},$ Lemma \ref{lem: r_max two sided} implies
    \[
        \frac{r_{\max}(s)}{\sqrt{2s + K a^{\frac{2}{\gamma\alpha}}}}
        \ge \frac{\sqrt{2s-C}}{\sqrt{2s + s/K}} \ge 1-\theta,
    \]
    if $K$ is sufficiently large (depending on $C$). Let us now prove the upper bound.
    
    Since 
    $r^2_{\max}(s)-2s\le O(s^{1-\alpha})$,  and $r^2_{\max}(s)\ge 2s$, we have
    \[
        r^2_{\max}(s)-2s \le O(r^{2-2\alpha}_{\max}(s)).
    \]
    Because $\alpha>c>0$, by choosing $K$ large enough we have
    \[
        \frac{r_{\max}^2(s)-2s}{ r_{\max}^2(s) }
        \le \frac{K^{\alpha}}{100^{1-\alpha} r_{\max}^{2\alpha}(s)}
    \]
    Now we can apply Young's inequality $x+y\ge x^{\alpha}y^{1-\alpha}$ to obtain
    \[
        \frac{K a^{\frac{2}{\gamma \alpha}}}{r^2_{\max}(s)}+ \frac{1}{100}a^{-2} \ge \frac{K^{\alpha}}{100^{1-\alpha}r^{2\alpha}_{\max}(s)} a^{\frac{2}{\gamma }-2(1-\alpha)}
        \ge \frac{r_{\max}^2(s)-2s}{ r_{\max}^2(s) }.
    \]
    It follows that if $s\ge K^2 a^{\frac{2}{\gamma \alpha}}$,
    \begin{align*}
        \frac{2s + K a^{\frac{2}{\gamma\alpha}}}{r_{\max}^2(s)} - 1 + \frac{1}{100}a^{-2}
        = \frac{K a^{\frac{2}{\gamma\alpha}}}{r_{\max}^2(s)} + \frac{1}{100}a^{-2} - \frac{r_{\max}^2(s)-2s}{ r_{\max}^2(s) }
        \ge 0
    \end{align*}
    Thus,
    \[
        \frac{r_{\max}(s)}{\sqrt{2s + K a^{\frac{2}{\gamma\alpha}}}}
        \le \left(1 - \frac{1}{100}a^{-2}\right)^{-\frac{1}{2}} \le 1 + \frac{1}{100}a^{-2}.
    \]
    \end{proof}
    The rest of the proof follows similar to \cite[Proposition 4.2]{ABDS22}.
    Since we are assuming $1/\gamma \ge \frac{3}{2}\alpha$, we can guarantee
    \[
        K^2 a^{\frac{2}{\gamma\alpha}} \ge K^2 a^{3} \ge Da^{3},
    \]
    by taking $K^2\ge D,$ where $D$ is given by Proposition \ref{prop: Fz bdd by barrier}.  Hence, we can apply the barrier argument to conclude
    \[
        F_z^2(z,s)\le \psi_a\left( \frac{F(z,s)}{\sqrt{2s + K a^{\frac{2}{\gamma\alpha}}}} \right) \le Ca^{-2},
    \]
    whenever $s\ge K^2 a^{\frac{2}{\gamma\alpha}}$ and $F(z,s) \ge \sqrt{s}.$ 
    The conclusion then follows by putting $s =K^2 a^{\frac{2}{\gamma\alpha}}.$
\end{proof}

 Here, a major difference from \cite{ABDS22} is that $\alpha$ cannot be any number close to $1$. In order for Proposition \ref{prop: F_z gamma alpha} to imply an improvement in estimates, we need to have $\gamma\geq 1$. The issue is the restriction $s\geq Da^3$ in our barrier, which forces $1\geq 1/\gamma \geq \frac{3}{2}\alpha$. Thus, we can only consider $\alpha\in (0,2/3]$.

\begin{Proposition}
    Suppose \eqref{ineq: star} holds for some $\alpha\in (0,2/3].$
    If $(q_0,s_0)$ satisfies $F\ge \sqrt{2s_0}$ and $s_0$ is sufficiently big, we have
    \[
        -(F^2)_{zz}
        \le C s_0^{-(1+\frac{\alpha^2}{200})\alpha},
    \]
    at $(q_0,s_0)$.
\end{Proposition}
\def \TF {\widetilde{F}}
\def \TEE {\widetilde{\mathcal{E}}}
\begin{proof}
Fix $1\le \gamma \le \frac{1}{\max(\frac{3}{2}\alpha,1-\alpha)}.$
    We denote by $\TF(z,t)$ the radius of the sphere of symmetry that has signed distance $z$ from the point $q_0$.  The function $\TF(z,t)$ satisfies the evolution equation
    \begin{align*}
     -\TF_s
     &= \TF_{zz} - \TF^{-1}(1+\TF_z^2)
     + 2\TF_z \left\{(\ln \TF)_z(0,s)
      - \int_0^z (\ln \TF)^2_z(\zeta,s) \rd \zeta\right\}\\
    &\qquad 
    - \TF^{-1}\TEE^{\rm orb}
    + \TF_z\int_0^z \TEE^{\rm rad}(\zeta,s)\,\rd \zeta.
 \end{align*}
 Similar to Lemma \ref{lem: Eorb>0},
    $\TEE^{\rm orb} >0$, whenever $\TF(z,s)\ge \sqrt{s}.$
In particular, since $\TF_{zz}<0,$
\[
   -\TF_s(0,s)
     \le   - \TF^{-1}(0,s)(1-\TF_z^2(0,s)).
\]
By Proposition \ref{prop: F_z gamma alpha}, 
\[
    \tfrac{1}{2} \partial_s \TF^2(0,s)
    \ge 1-\TF_z^2(0,s)
    \ge 1 - C\,s^{-\gamma\alpha}
    \ge 1- C\,s^{-\alpha}
    > \tfrac{1}{2},
\]
whenever $\TF(0,s)\ge \sqrt{s}, s\ge \underline{s}.$ Since $\TF(0,s_0)\ge \sqrt{2s_0},$ a standard continuity argument gives
\[
    \TF^2(0,s)\ge 2s(1-Cs^{-\alpha}),
\]
if $s\ge s_0\ge  \underline{s}.$
In the following, we put $\eps := \frac{\alpha^2}{100}.$ 
By Proposition \ref{prop: F_z gamma alpha}, for $s\ge s_0,|z|\le s^{\frac{1+\eps}{2}}$,
\[
    \TF^2(z,s)
    \ge \TF^2(0,s) - 
    C s^{\frac{1}{2}+\frac{1+\eps}{2}-\gamma\alpha}
    \ge 2s(1-Cs^{-\frac{\alpha}{8}}).
\]
On the other hand, the
condition \eqref{ineq: star} gives
\[
    \TF^2(z,s)\le 2s(1+Cs^{-\alpha}).
\]

\def \TH {\widetilde{H}}

We next consider the  parabolic cylinder
\[
    Q:= \left[-s_0^{\frac{1+\eps}{2}},s_0^{\frac{1+\eps}{2}}\right]
    \times [s_0, s_0 + s_0^{1+\eps}].
\]
Define
\[
    \TH(z,s) := \tfrac{1}{2}\TF^2(z,s) -s.
\]
By the assumption, $\TH(0,s_0)\ge 0.$ Moreover, the preceding arguments imply that we
can find a positive constant $L$ such that
\[
    -L s^{1-\frac{\alpha}{8}}
    \le \TH(z,s) \le L s^{1-\alpha},
\]
in $Q.$ In particular,
\[
    -L (2s_0)^{(1+\eps)(1-\frac{\alpha}{8})}
    \le \TH(z,s) \le L s^{(1+\eps)(1-\alpha)},
\]
in $Q.$ $\TH$ satisfies
\[
    -\TH_s - \TH_{zz} = -S,
\]
where
\begin{align*}
    S(z,s) &= 2\TF^2_{z}
    -  2\TF\TF_z \left\{(\ln \TF)_z(0,s)
      - \int_0^z (\ln \TF)^2_z(\zeta,s) \rd \zeta\right\}\\
    &\qquad 
    - \TEE^{\rm orb}
    + \TF\TF_z\int_0^z \TEE^{\rm rad}(\zeta,s)\,\rd \zeta.
\end{align*}
By Proposition \ref{prop: F_z gamma alpha} and Lemma \ref{lem: error by F} we have
\[
    |S| \le C s^{-\gamma\alpha} \le Cs_0^{-\gamma\alpha},
\]
in $Q.$
The rest of the proof follows similar to the three-dimensional case. See \cite[Proposition 4.4]{ABDS22}.
\end{proof}

\begin{Corollary}
\label{cor: improve alpha}
    Suppose that \eqref{ineq: star} holds for some $\alpha\in (0,2/3].$ 
    If $0<\tilde{\alpha}<(1+\frac{\alpha^2}{200})\alpha,$
    then ${\rm (\star_{\tilde{\alpha}})}$ holds. 
\end{Corollary}
\begin{proof}
    For any $s\ge\underline{s}$, consider the point $z^*$ where $F(\cdot,s)$ attains its maximum. At $(z^*,s), F=r_{\max}(s)\ge \sqrt{2s}, F_z=0$, and 
    \[
        -(F^2)_{zz}
        \le C s^{-(1+\frac{\alpha^2}{200})\alpha},
    \]
    by the previous Proposition.
    Using the evolution equation of $F$,
    at $(z^*,s)$ we have
    \[
        \tfrac{1}{2}\tfrac{\partial}{\partial s} F^2
        = -(F^2)_{zz} + 1 + \bEE^{\rm orb} 
        \le 1 + C s^{-(1+\frac{\alpha^2}{200})\alpha}. 
    \]
    So ${\rm (\star_{\tilde{\alpha}})}$ holds by integration.
\end{proof}

\begin{Lemma}
\label{lem: alpha>1/2}
    \eqref{ineq: star} holds for $ \alpha = 2/3.$
\end{Lemma}
\begin{proof}
    As we have discussed at the beginning of this section, $\rho_{\max}(\tau)\le O(e^{\frac{\tau}{16}}).$
    So \eqref{ineq: star} holds for $\alpha=\frac{1}{16}.$ 
    Iterating Corollary \ref{cor: improve alpha} finitely many times, we get that \eqref{ineq: star} holds for $ \alpha = 2/3.$
\end{proof}

\begin{Corollary}
\label{cor: diam lower}
    Suppose $\Gamma^+$ dominates. Then
    \[
        {\rm diam}(\bar g_s) \ge c \, s^{\frac{5}{6}},
    \]
    for some constant $c>0$ and large $s$.
\end{Corollary}
\begin{proof}
By Lemma \ref{lem: alpha>1/2}, \eqref{ineq: star} holds for $\alpha = 2/3,$ and we  take $\gamma = \frac{2}{3\alpha}=1.$ 
    Let $z^*$ be the maximum point of $F(\cdot,s)$ and let $z<z^*$ be the point such that $F(z,s)=\sqrt{s}$. By Proposition \ref{prop: F_z gamma alpha},
    \begin{align*}
        \sqrt{2s}-\sqrt{s}
        &\le F(z^*,s)-F(z,s)
        = \int_{z}^{z^*} F_z(\zeta,s)\,\rd \zeta\\
        &\le C\,(z^*-z)s^{-\frac{\gamma\alpha}{2}}
        \le C\, {\rm diam}(\bar g_s)s^{-\frac{1}{3}}.
    \end{align*}   
    The conclusion follows.
\end{proof}

Finally, we  show that the positive modes cannot dominate under our curvature decay assumption \eqref{eq-A4}.

\begin{Proposition}
    Let $R$ denote the scalar curvature of $M^4$, and let $o$ be the (unique) critical point of the potential $f$. Suppose
    \[
        \sup_{\partial B_r(o)} R = O(r^{-\eta}),
    \]
    as $r\to \infty$, for some $\eta\in (\frac{1}{3},1).$
    Then $\Gamma^+$ cannot dominate.
\end{Proposition}
\begin{proof}
      Similar to Perelman's estimates in \cite[Section 8]{Per02}, we have
    \[
        \partial_s d_{\bar g_s}(x,y)
        \le C\sqrt{R_{\max}(s)},
    \]
    where $R_{\max}(s):= \max_{\Sigma_s}R$. 
    Since $\Ric\ge 0$, for any $x\in \Sigma$ 
    \[
        \partial_s R(\chi_s(x)) = \langle \nabla R, \tfrac{\nabla f}{|\nabla f|^2} \rangle
        = -\tfrac{2}{|\nabla f|^2} \Ric(\nabla f, \nabla f)\le 0,
    \]
    and thus
    $R_{\max}(s)$ is decreasing in $s.$
    By \cite[Theorem 2.1]{CDM22}, $f(x)\ge c\,d_g(o,x)$  for some $c>0$. So for any $x\in \Sigma_s, d_g(o,x)\le Cs,$ and thus
    \[
        R_{\max}(s) = \max_{\Sigma_s}R \le O(s^{-\eta}),
    \] 
    for large $s.$
    So, by integration, for $s \ge \underline{s}$,
    \[
        {\rm diam}(\bar g_s)
        \le {\rm diam}(\bar g_{\underline{s}})
        + C \int_{\underline{s}}^{s} \sqrt{R_{\max}(\sigma)}\, \rd\sigma
        \le C s^{1-\frac{\eta}{2}},
    \]
    which contradicts Corollary \ref{cor: diam lower}. Hence, $\Gamma^+$  cannot dominate.
\end{proof}




\section{Asymptotics In The Case When Neutral Mode Dominates}
\label{sec-neutral}

\def \DD {\mathcal{D}}

In this section, we consider the first possibility in Proposition \ref{prop-merlezaag}. Thus, throughout this section we assume
\begin{equation}
\label{eq-prevail-mode}
   \Gamma^+(\tau)+\Gamma^-(\tau)
        \le o(1)\Gamma^0(\tau). 
\end{equation}

\subsection{Asymptotics In The Cylindrical Region}
Under assumption \eqref{eq-prevail-mode},  we obtain precise asymptotics of $G$ in the cylindrical region following arguments similar to \cite[Section 5]{ABDS22}. Following \cite{ABDS22}, we define
\[
    \|u\|_{\DD}^2
    := \int (u^2 + u_\xi^2)\, \rd\nu
    = \frac{1}{\sqrt{4\pi}}
     \int \left(u^2 + u_\xi^2\right)
     e^{-\frac{|\xi|^2}{4}}\, d\xi.
\]
Note that if $u$ has compact support, then
\[
    \|u\|_{\DD}^2
    = \langle u, (2-\LL) u \rangle_{\HH}.
\]

As before, we consider the spectral decomposition for the operator $\LL$, and we let $\HH_0$ be the subspace spanned by $h_2(\xi)= \frac{1}{2\sqrt{2}}(\xi^2-2)$, the unit eigenfunction corresponding to eigenvalue $0$.
We consider the projection of $\HG$ to $\HH_0$, and write
\[
    P_0\HG(\xi,\tau)
    = \alpha(\tau)h_2(\xi),
\]
where
\[
    \alpha(\tau):=
    \langle \HG(\cdot,\tau), h_2 \rangle_{\HH}.
\]
Let
\[
    A(\tau) := \sup_{\sigma\le\tau} |\alpha(\sigma)|.
\]
Clearly, $\frac{1}{C}A^2
\le \Gamma^0 \le CA^2$, for some constant $C$. Hence, assumption \ref{eq-prevail-mode} implies $\frac{1}{C}A^2
\le \Gamma \le CA^2$. By the proof of Lemma \ref{lem: ode for modes}, for $\tau\le \bar\tau$ we have
\[
    \delta(\tau)
    \le \Gamma^{\frac{1}{80}}(\tau)
    \le C A^{\frac{1}{40}}(\tau).
\]

Similar to \cite{ABDS22} we can show
\begin{gather*}
        \|P_+\HG(\cdot,\tau)\|_{\HH}
        \le o(1)A(\tau)\\
        \|P_-\HG(\cdot,\tau)\|_{\HH} +   \|P_-\HG(\cdot,\tau)\|_{\DD}
        \le C \delta^{\frac{1}{400}}(\tau)A(\tau)
\end{gather*}
Note that in the proof of the last inequality we have used the estimate
    \[
        \int (\partial_\tau-\LL)\widehat{G}\, \rd\nu
        \le C \delta^{\frac{1}{100}}(\tau)\Gamma(\tau)
        \le C \delta^{\frac{1}{100}}(\tau)A^2(\tau),
    \]
which holds for $\tau\le\bar\tau$ according to \ref{E by Gamma} and the above estimates.

\begin{Lemma}
\label{lem: pos neg modes}
As $\tau \to -\infty$, we have
\[
    \int_{\IR} (1+|\xi|)^4
    \left|\HG(\xi,\tau)-\alpha(\tau)h_2(\xi)\right|^2
    \, \rd\nu(\xi)
    \le o(1)A^2(\tau),
\]
and
\[
    \int_{\IR} (1+|\xi|)^4
    \left|\partial_\xi\left(\HG(\xi,\tau)-\alpha(\tau)h_2(\xi)\right)\right|^2
    \, \rd\nu(\xi)
    \le o(1)A^2(\tau).
\]
\end{Lemma}
\begin{proof}
    The  proof is similar to the three-dimensional case, c.f. \cite[Lemma 5.4]{ABDS22}.
\end{proof}

\begin{Lemma}
\label{lem: G_xi at 0}
    As $\tau \to -\infty$, we have
    $|\HG_{\xi}(0,\tau)|\le o(1)A(\tau).$
\end{Lemma}
\begin{proof}
    The  proof is similar to the three-dimensional case, c.f. \cite[Lemma 5.5]{ABDS22}.
\end{proof}

\begin{Proposition}
    Let $E=(\partial_\tau-\LL)G$.
    Then as $\tau \to -\infty$, we have
\[
    \int_{I_\tau} E(\cdot,\tau) h_2\, \rd\nu
    = -2\alpha^2(\tau) + o(A^2(\tau))+O(e^\tau),
\]
where $I_\tau= [-\delta^{-\frac{1}{100}}(\tau),\delta^{-\frac{1}{100}}(\tau)].$

\end{Proposition}
\begin{proof}
The proof is similar to \cite[Proposition 5.6]{ABDS22}, and we only need to control the extra error terms. By Proposition \ref{prop: G eqn},
\begin{align*}
     E
    &= - G + \tfrac{1}{2}(\sqrt{2}+G)
    - (\sqrt{2}+G )^{-1}(1+G_{\xi}^2)
    + 2G_\xi \left\{ 
    \frac{G_{\xi}}{\sqrt{2}+G}(0,\tau)
        - \int_0^{\xi}\frac{G_{\xi}^2}{(\sqrt{2}+G)^2}(x,\tau)\, \rd x\right\}\\
    &\qquad - (\sqrt{2}+G)^{-1} \EE^{\rm orb}
    + G_{\xi} \int_0^\xi \EE^{\rm rad}(x,\tau)\, \rd x\\
    &= -\tfrac{1}{2\sqrt{2}}G^2
    -\tfrac{1}{\sqrt{2}}G_\xi^2
    +E_1+E_2+E_3+E_4,
\end{align*}
where $E_1, E_2, E_3$ are defined similar to \cite{ABDS22}, and $E_4$ is defined as
\begin{align*}
    E_4 = - (\sqrt{2}+G)^{-1} \EE^{\rm orb}
    + G_{\xi} \int_0^\xi \EE^{\rm rad}(x,\tau)\, \rd x.
\end{align*}
The terms $E_1, E_2$ and $E_3$ can be treated the same way as in \cite{ABDS22}. Regarding our extra error term $E_4$, for $\tau\le\bar\tau$ and $\xi\in I_\tau$ we have
\[
    |E_4|(\xi,\tau)\le C e^\tau + |G_\xi|e^\tau |\xi|.
\]
Hence, by Lemma \ref{lem: pos neg modes},
\[
    \left|\int_{I_\tau}
    E_4 h_2 \, \rd\nu\right|
    \le C e^\tau
    + C e^\tau \int_{I_\tau}
    |G_\xi| |\xi|(\xi^2+2)\, \rd\nu
    \le C e^\tau 
    + Ce^{\tau}o(1)A(\tau)
    \le  C e^\tau.
\]
The conclusion follows by combining the estimates similar to \cite{ABDS22}.
\end{proof}

\begin{Corollary}
    For $\tau\le\bar\tau,$
    \[
        \alpha'(\tau)=-2\alpha^2(\tau)+o(A^2(\tau)).
    \]
\end{Corollary}
\begin{proof}
    Directly by the previous Proposition, we have
    \[
    \alpha'(\tau)=-2\alpha^2(\tau)+o(A^2(\tau)) + O(e^\tau).
    \]
    It suffices to show that $e^\tau \le o(1)A^2(\tau).$

{
    Since we are assuming the neutral mode dominates,  for $\tau\le\bar\tau$ we have
    \[
        \Gamma^+(\tau)+\Gamma^-(\tau) \le o(1)\Gamma^0(\tau) \le o(1)A^2(\tau)\Rightarrow
        \Gamma(\tau) \le (1+o(1)) \Gamma^0(\tau).
    \]
    For any $\beta\in (0,1),$ pick $\eps>0$ satisfying $1-\eps = e^{-\beta}$.
    By Lemma \ref{lem: ode for modes}, for $\tau\le \bar \tau,$
    \[
        \Gamma^0(\tau)\ge (1-\eps)\Gamma^0(\tau+1)
        \ge (1-\eps)^2\Gamma^0(\tau+2)
        \ge \cdots
        \ge (1-\eps)^k \Gamma^0(\tau_*),
    \]
    where $\tau+k=\tau_*\in [\bar\tau-1,\bar\tau].$ Thus, for $\tau\le 2\bar\tau,$
    \[
         \Gamma^0(\tau) \ge  e^{\beta(\tau-\bar\tau)} \min_{[\bar\tau-1,\bar\tau]}\Gamma^0
         \ge c\, e^{\beta\tau},
    \]
    for some constant $c>0.$
    By the discussions at the beginning of this section, $\Gamma^0(\tau)\le C A^2(\tau),$ and thus,
    \[
        A^2(\tau)\ge c\, \Gamma^0(\tau) \ge c\, e^{\beta\tau}.
    \]
    It follows that $e^\tau\le o(1)A^2(\tau),$ and
    \[
        \alpha'(\tau)=-2\alpha^2(\tau)+o(A^2(\tau)).
    \]

}

\end{proof}


\begin{Corollary}
    As $\tau\to -\infty,$ 
    \[
        \alpha(\tau) = \frac{1}{(2+o(1))\tau}.
    \]
\end{Corollary}
\begin{proof}
The proof is similar to the three-dimensional case, c.f. \cite[Corollary 5.8, Corollary 5.9]{ABDS22}.
\end{proof}

\begin{Proposition}
\label{prop: G conv}
    As $\tau\to -\infty,$
    \[
        -\tau \, G(\xi,\tau)
        \to -\frac{1}{4\sqrt{2}}
        (\xi^2-2),
    \]
    in $C_{\rm loc}^\infty.$
\end{Proposition}
\begin{proof}
    The proof is similar to the three-dimensional case, c.f. \cite[Proposition 5.10]{ABDS22}.
\end{proof}

\begin{Corollary}
    As $\tau\to-\infty,$
    the domain of the function $G(\cdot,\tau)$ is an interval of length at most $o(-\tau).$
\end{Corollary}
\begin{proof}
    The proof is similar to the three-dimensional case, c.f. \cite[Corollary 5.11]{ABDS22}, but we shall include some more details.

    Let $(-d_\tau^-,d_\tau^+)$ be the domain of $G(\cdot,\tau)$.
    Recall that $G(\cdot,\tau)$ is concave. For any large constant $L \gg 1$, as $\tau\to -\infty$ we have
    \begin{align*}
        \frac{2-L^2+o(1)}{4\sqrt{2}}
        &\ge 
        |\tau|\,G(L,\tau) \ge 
        \left(1-\frac{2L}{d_\tau^+}\right) |\tau|\,G(0,\tau)
        + \frac{2L}{d_\tau^+} |\tau|\,G(d_\tau^+/2,\tau)\\
        &\ge \frac{2-o(1)}{4\sqrt{2}}\left(1-\frac{2L}{d_\tau^+}\right)
        - \sqrt{2}|\tau|\frac{2L}{d_\tau^+}.
    \end{align*}
    Hence
    \[
     -L^2+2+o(1)
    \ge  
    2-o(1)
    - \frac{2L}{d_\tau^+}
    \left(8|\tau|+2+o(1)\right),
    \]
    which implies
    \[
        d_\tau^+
        \le 
        \frac{2L}{L^2-o(1)}\left(8|\tau|+2+o(1)\right).
    \]
    Therefore, since $L$ can be arbitrarily large, $d_\tau^+\le o(-\tau)$. Inequality $d^-_\tau\le o(-\tau)$ can be similarly shown.
\end{proof}

\begin{Corollary}
\label{cor: max pt}
    Let $\xi_\tau^*$ be the unique point in space where $G(\cdot,\tau)$ attains its maximum. $\xi_\tau^* \to 0$ as $\tau\to-\infty$.
\end{Corollary}

\subsection{Asymptotics in the intermediate region}

We now describe the asymptotics in the intermediate region, defined by  $|z|\ge M\sqrt{s}$ and $F(z,s)\ge \theta\sqrt{2s}$,
for some large constant $M$ and some small constant $\theta$.
We closely follow \cite[Section 6]{ABDS22}.

\begin{Proposition}
\label{prop: Fz intermediate}
    Fix a small number $\theta\in (0,1/2)$ and a large number $M\ge 10.$
    For $s\ge \underline{s}(\theta,M),$
    \[
        F_z^2(z,s) \le \frac{M^2+C(\theta)}{M^2-2} \frac{1}{2\log s} 
        \left( \frac{2s}{F^2(z,s)}-1 \right),
    \]
    whenever $|z|\ge M\sqrt{s},$ and $F(z,s)\ge \theta\sqrt{2s}.$
\end{Proposition}
\begin{proof}
    The proof follows verbatim as in the three-dimensional case, c.f. \cite[Proposition 6.1]{ABDS22}.
    The only place that requires justification is where we apply the barrier argument in Proposition \ref{prop: Fz bdd by barrier}, when $\log s\ge \frac{M^2+2}{M^2-2}\frac{a^2}{2}, |z|\ge M\sqrt{s},$ and $F(z,s)\ge \theta\sqrt{2s}.$
    Since $s=e^{\log s}\ge e^{a^2/4}\gg a^3,$ we can apply Proposition \ref{prop: Fz bdd by barrier} with $\mu=0$. 
\end{proof}

\begin{Corollary}
    Fix a small constant $\theta\in (0,1/2).$ If $s\ge\underline{s}(\theta),$
    \[
        \partial_s F^2(z,s)
        \ge 2 - \frac{C(\theta)}{\log s},
    \]
    whenever $|z|\ge 10\sqrt{s}$ and $F(z,s)\ge \theta\sqrt{2s}.$
\end{Corollary}
\begin{proof}
    The proof is similar to the three-dimensional case, c.f. \cite[Corollary 6.2]{ABDS22}. 
    By Corollary \ref{cor: max pt}, 
    \[
        F_z(z,s)> 0,\quad \text{for } z\le -\sqrt{s};\qquad
        F_z(z,s)< 0,\quad \text{for } z\ge \sqrt{s}.
    \]
    Recall that by Lemma \ref{lem: Eorb>0}, $F_{zz}\le 0,\bEE^{\rm orb}\ge 0,\bEE^{\rm rad}\ge 0,$ for large $s.$
    By the evolution equation for $F$, we have
    \begin{align*}
        -F_s &= F_{zz} - F^{-1}(1-F_z^2) - 2F_z \int_0^{z} \frac{F_{zz}}{F}(\zeta,s)\,{\rm d}\zeta\\
        &\qquad - F^{-1}\bEE^{\rm orb}
        + F_z\int_0^z \bEE^{\rm rad}(\zeta,s)\,{\rm d}\zeta\\
        &\le- F^{-1}(1-F_z^2),
    \end{align*}
    for $|z|\ge \sqrt{s}.$ Applying Proposition \ref{prop: Fz intermediate} with $M=10,$
    \[
        F_z^2(z,s) \le C(\theta)/\log s,
    \]
    whenever $|z|\ge 10\sqrt{s}$ and $F(z,s)\ge \theta\sqrt{2s}.$
    The assertion follows.
\end{proof}

\begin{Proposition}
\label{prop: F lower}
    Fix a small constant $\theta\in (0,\frac{1}{2})$ and a large number $M\ge 20.$
    If $s\ge\underline{s}(\theta,M),$ then
    \[
        F^2(z,s) \ge 2s - \frac{M^2+C(\theta)}{M^2-2} \frac{z^2}{2\log s},
    \]
    whenever $|z|\ge M\sqrt{s}$ and $F(z,s)\ge \theta\sqrt{2s}.$
\end{Proposition}
\begin{proof}
    The proof follows verbatim as in the three-dimensional case, c.f. \cite[Proposition 6.3]{ABDS22}.
\end{proof}

\begin{Proposition}
\label{prop: F upper}
    Fix a small constant $\theta\in (0,\frac{1}{2})$ and a large number $M\ge 20.$
    If $s\ge\underline{s}(\theta,M),$ then
    \[
        F^2(z,s)\le 2s - \frac{M^2-C(\theta)}{M^2} \frac{z^2}{2\log s},
    \]
    whenever $|z|\ge M\sqrt{s}$ and $F(z,s)\ge \theta\sqrt{2s}.$
\end{Proposition}
\begin{proof}
    The proof follows verbatim as in the three-dimensional case, c.f. \cite[Proposition 6.4]{ABDS22}.
\end{proof}
As in \cite{ABDS22}, combining Proposition \ref{prop: F lower} and Proposition \ref{prop: F upper} and sending $M\to \infty$, we have the following.

\begin{Corollary}
\label{cor: bar z_i}
    Fix a small constant $\theta\in (0,\frac{1}{2})$.
    If $s\ge\underline{s}(\theta),$ we have
    \[
        \left\{z: F(z,s)\ge \theta\sqrt{2s} \right\}
        = [-\bar z_1(\theta,s), \bar z_2(\theta,s)],
    \]
    and
    \begin{align*}
        \bar z_1(\theta,s) &= (2+o(1))\sqrt{1-\theta^2}\sqrt{s\log s},\\
        \bar z_2(\theta,s) &= (2+o(1))\sqrt{1-\theta^2}\sqrt{s\log s}.
    \end{align*}
\end{Corollary}

\begin{Corollary}
     Fix a small constant $\theta\in (0,\frac{1}{2})$.
    If $s\ge\underline{s}(\theta),$ we have
    \[
        F^2(z,s) = 2s - \frac{z^2}{2\log s} + o(s),
    \]
    for $|z|\le 2\sqrt{1-\theta^2}\sqrt{s\log s}.$
\end{Corollary}

\subsection{Asymptotics in the Tip Region}

Following \cite[\S 7]{ABDS22}, we analyze the asymptotics near each tip. 
Throughout this section, we assume that the neutral mode dominates.
For large $s,$ the function $z\mapsto F(z,s)$ is defined on the interval $[-d_{{\rm tip},1}(s),d_{{\rm tip},2}(s)]$, where $d_{{\rm tip},1}(s)$ and $d_{{\rm tip},2}(s)$ denote the distance of the reference point $q$ from each tip measured with $\bar g_s.$
We first derive the asymptotics of $d_{{\rm tip},i}(s)$.

\begin{Proposition}
For $i=1,2,$
    \[
        \lim_{s\to \infty} \frac{d_{{\rm tip},i}(s)}{\sqrt{s\log s}}
        = 2.
    \]
\end{Proposition}
\begin{proof}
   The proof follows verbatim as in the three-dimensional case, c.f. \cite[Proposition 7.1]{ABDS22}, because it only relies on Corollary \ref{cor: bar z_i} and the concavity of $F.$
\end{proof}


We then analyze the asymptotic behaviors of the scalar curvature at each tip.
For simplicity, we assume that $R$ decays uniformly and we believe that the conclusions should still hold without assuming the decay. We write $$R_{{\rm tip},i}(s) = R(p_s^i),$$
where recall that $p_s^1,p_s^2$ denote the two tips of $\Sigma_s$ given by Corollary \ref{cor: qualitative}.

\begin{Lemma}
    Suppose that $R$ decays uniformly. Then for $i=1,2,$
    \[
        \tfrac{\rm d}{{\rm d}s} d_{{\rm tip},i}(s)
        = (1+o(1)) R_{{\rm tip},i}^{\frac{1}{2}}(s).
    \]
\end{Lemma}
\begin{proof}
    If $R\to 0$ at infinity, $(\Sigma,  R(p^i_s) \bar g_s,p^i_s)$ converges to the three-dimensional Bryant soliton as $s\to \infty.$
    The proof follows verbatim as in the three-dimensional case, c.f. \cite[Lemma 7.3]{ABDS22}.
\end{proof}

\begin{Proposition}
Suppose that $R$ decays uniformly.
    Then for $i=1,2,$ as $s\to \infty,$
    \[
        R_{{\rm tip},i}(s) = (1+o(1)) \frac{\log s}{s}.
    \]
\end{Proposition}
\begin{proof}
   The proof follows verbatim as in the three-dimensional case, c.f. \cite[Proposition 7.4]{ABDS22}.
\end{proof}

%
%
%
%
%

\section{Asymptotics When the Positive Modes Dominate}
\label{sec-ref-symm}

We now consider the second case in Proposition \ref{prop-merlezaag}. More specifically, we assume
\[
    \Gamma^0(\tau) + \Gamma^-(\tau)
    = o(1)\, \Gamma^+(\tau),
\]
as $\tau\to -\infty$. 

In addition to $(M^4, g, f)\in \MM$, we assume that $(M,g)$ has reflexive symmetry, i.e., $F(\cdot,s)$ is an even function.

Under the assumptions above we give precise asymptotics of our solution in the cylindrical region. The hope is that we would be able to use this asymptotics to eventually to exclude the case of positive mode dominating at all. 
\begin{Theorem}
\label{thm: Z2}
When the positive modes dominate and $(M,g)$ is reflexive symmetric, 
\[
    F(z,s) = \frac{\log s}{\sqrt{2} s} + o\left(\frac{\log s}{s}\right),
\]
uniformly on compact sets of the form $|z|\le L\sqrt{s}.$
\end{Theorem}

By Lemma \ref{lem: ode for modes},
\[
    \Gamma^+(\tau-1) \le e^{-1}\Gamma^{+}(\tau)
    + C \delta^{\frac{1}{200}}(\tau) \Gamma^+(\tau).
\]
Iterating this yields
\[
    \Gamma^+(\tau) \le O\left(e^{\frac{\tau}{2}}\right).
\]
Thus, as a rough estimate we get $\Gamma^+(\tau)\le O\left(e^{\frac{\tau}{2}}\right)$.

We need the following coarse estimate.
\begin{Lemma}
\label{lem: delta by rho}
For $\tau\le \bar\tau,$
    \[
        \rho_{\max}(\tau)\ge G(0,\tau)
        \ge - Ce^{\frac{\tau}{2}}
        - C\rho\left(\tfrac{\tau}{2}\right).
    \]
    As a consequence,
    \[
        \delta(\tau) 
        \le Ce^{\frac{\tau}{2}} + C\rho\left(\tfrac{\tau}{2}\right).
    \]
\end{Lemma}
\begin{proof}
Since $0$ is the reference point, we have that
for any $\eps>0,$ if $\tau\le \bar\tau(\eps),$ then
$|G(0,\tau)|\le \eps.$
    By the evolution equation of $G$ given in Proposition \ref{prop: G eqn}, at $(0,\tau),$
    \begin{align*}
        G_\tau
        &\le \tfrac{1}{2}(\sqrt{2}+G) - (\sqrt{2}+G)^{-1}
        + (\sqrt{2}+G)^{-1}G_{\xi}^2
        - (\sqrt{2}+G)^{-1}\EE^{\rm orb},\\
        \partial_\tau (\sqrt{2}+G)^2
        &\le (\sqrt{2}+G)^2 - 2 + C\rho(\tau),
    \end{align*}
    where we applied Lemma \ref{lem: pos Eorb} and Proposition \ref{prop: G_xi bdd by barrier}. 
    By integration, if $\tau\le 2\bar\tau(\eps),$
    \begin{align*}
        e^{-\tau} \left((\sqrt{2}+G)^2 - 2\right)(0,\tau)
        &\ge -10\eps e^{-\frac{\tau}{2}}
        - C\int_{\tau}^{\frac{\tau}{2}} e^{-\sigma} \rho(\sigma)\, d\sigma\\ 
        &\ge -e^{-\frac{\tau}{2}}
        - C\rho\left(\tfrac{\tau}{2}\right)(e^{-\tau}-e^{-\frac{\tau}{2}}).
    \end{align*}
    Then
    \begin{align*}
        (\sqrt{2}+G)^2(0,\tau)
        &\ge 2
        -e^{\frac{\tau}{2}}
        - C\rho\left(\tfrac{\tau}{2}\right),\\
        G(0,\tau) &\ge -\sqrt{2}
        + \sqrt{2
        -e^{\frac{\tau}{2}}
        - C\rho\left(\tfrac{\tau}{2}\right)}
        \ge -C e^{\frac{\tau}{2}}
        - C\rho\left(\tfrac{\tau}{2}\right).
    \end{align*}
Thus, the conclusion follows.
        
\end{proof}

It follows by Lemma \ref{lem: delta by rho} that
\[
    \rho_{\max}(\tau) \le O\left(e^{\frac{\tau}{16}}\right),\quad
    \rho(\tau)\le O\left(e^{\frac{\tau}{16}}\right),\quad
    \delta(\tau) \le 
    \rho(\tau/2)
    \le O\left(e^{\frac{\tau}{32}}\right).
\]

Let $h_{k}$ be the eigenfunctions of $\LL$ corresponding to eigenvalues $1-\frac{k}{2},k\in \mathbb{N}.$ 
Write
\[
    \HG(\xi,\tau)
    = \sum_{k\ge 0} a_{k}(\tau) h_{k}(\xi),
\]
where
\[
    a_{k}(\tau) = \langle \HG(\cdot,\tau), h_{k}\rangle_{\HH}.
\]

Since we assumed reflexive symmetry,
\[
    a_1(\tau)
    = \langle \widehat{G}(\cdot,\tau), h_1 \rangle_{\HH}
    \equiv 0,\qquad
    P_+\widehat{G}
    = a_0(\tau).
\]
Finally we conclude the following.
\begin{Lemma}
\[
    a_0(\tau) = - \tfrac{1}{\sqrt{2}} \tau e^{\tau} + o(\tau e^\tau).
\]
\end{Lemma}
\begin{proof}
    By Lemma \ref{lem: Eorb>0},
    \[
        \EE^{\rm orb} = (\tfrac{1}{2}+o(1)) e^{-\tau}(\sqrt{2}+G)^2R^2.
    \]
   Then it is not hard to see that, by Lemma \ref{lem: bound on parabolic op}, 
   \begin{align*}
       a_0'(\tau)
        &= \langle (\partial_\tau-\LL)\widehat{G}+\LL \widehat{G},h_0 \rangle_{\HH}\\
        &= a_0(\tau) - \langle (\sqrt{2}+G)^{-1}\EE^{\rm orb},h_0 \rangle
        + O(e^{\eps \tau}) a_0(\tau)\\
        &= a_0(\tau) - \left(\tfrac{1}{\sqrt{2}}+o(1)\right)e^{\tau} + O(e^{\eps \tau}) a_0(\tau),
   \end{align*}
   where $\eps>0$ is a small constant.
   By integration,
   \[
        a_0(\tau) = - \tfrac{1}{\sqrt{2}} \tau e^{\tau} + o(\tau e^\tau).
   \]
\end{proof}

Theorem \ref{thm: Z2} follows by the standard interpolation inequalities.

\bibliographystyle{amsalpha}

\newcommand{\alphalchar}[1]{$^{#1}$}
\providecommand{\bysame}{\leavevmode\hbox to3em{\hrulefill}\thinspace}
\providecommand{\MR}{\relax\ifhmode\unskip\space\fi MR }
\providecommand{\MRhref}[2]{%
  \href{http://www.ams.org/mathscinet-getitem?mr=#1}{#2}
}
\providecommand{\href}[2]{#2}

\end{document}